\documentclass[11pt]{article}
\usepackage{amsfonts, amssymb, a4, setspace, lscape,color}
\usepackage{amsmath}
\usepackage{newinutile}
\usepackage{graphicx,calc}
\usepackage{cite}
\usepackage{amsthm}

\theoremstyle{plain}
\ifx\thechapter\undefined
\newtheorem{thm}{Theorem}
\else

\fi
  \newtheorem{Definition}[thm]{Definition}
  \newtheorem{Remark}[thm]{Remark}
  \theoremstyle{plain}
  
  \theoremstyle{plain}
  \newtheorem{Lemma}[thm]{Lemma}
  \theoremstyle{plain}

\usepackage{mathrsfs}

\usepackage{color}
\usepackage{latexsym,amssymb,lastpage}
\usepackage{graphicx,amsfonts}
\usepackage{times,mathptmx,bm,amsmath}
\usepackage{dcolumn}

\numberwithin{equation}{section}

\begin{document}
\title{Weak solutions to 
Allen--Cahn--like equations modelling consolidation of porous media}
\author{{\sc Pietro Artale Harris}\\[2pt]
Dipartimento di Scienze di Base ed Applicate Per l'Ingegneria, Universit\`{a} di Roma La Sapienza, via A. Scarpa, 16, Roma, I.\\[6pt]
{\sc Emilio N. M. Cirillo}\\[2pt]
Dipartimento di Scienze di Base ed Applicate Per l'Ingegneria, Universit\`{a} di Roma La Sapienza, via A. Scarpa, 16, Roma, I.\\[6pt]
{\sc Adrian Muntean}\\[2pt]
Department of Mathematics and Computer Science, CASA -- Center for Analysis, Scientific computing and Applications, 
Institute for Complex Molecular Systems (ICMS), \\[6pt]
 Technische Universiteit Eindhoven,  P.O. Box 513,  5600 MB Eindhoven\\[6pt]
}
\pagestyle{headings}
\markboth{P. ARTALE HARRIS, E.M.N CIRILLO, A. MUNTEAN}{\rm Weak solutions to 
Allen--Cahn--like equations modelling consolidation of porous media}
\maketitle


\begin{abstract}
{We study the weak solvability of a system of  coupled Allen--Cahn--like equations resembling cross--diffusion which is arising as a model for the consolidation of saturated porous media. Besides using energy like estimates, we cast the special structure of the system in the framework of the Leray--Schauder fixed point principle and ensure this way the local existence of strong solutions to a regularised version of our system. Furthermore, weak convergence techniques ensure the existence of weak solutions to the original consolidation problem. The uniqueness of global-in-time solutions is guaranteed in a particular case. Moreover, we use a finite difference scheme to show the negativity of the vector of solutions.}{Weak solutions; cross--diffusion system;  energy method; Leray--Schauder fixed point theorem; finite differences; consolidation of porous media}
\end{abstract}

\section{Introduction}
\label{s:intro}
Porous solids with fluids moving inside  are very important to numerous engineering applications including the classical soil compaction and consolidation problem in civil engineering and poromechanics, or the biomechanics of bones and tissues, consolidation and subsidence control in environmental engineering, seepage of polluted liquids leaking from dangerous reservoirs, oil extraction plants and geothermal reservoirs; see for instance chapter 6 in \cite{Bear} for basic theoretical accounts and \cite{Higdon,Benalla}, \cite{CIS2011} and references 
cited therein for more modern applications. 

A typical unwanted phenomenon in the consolidation context is the 
occurrence of phase separation between fluid--rich and fluid--poor 
regions in porous media. 
Indeed, in such a case the porous medium, even in presence of an 
external pressure, could possibly have in its interior dangerous fluid 
bubbles \cite{CIS2010}. 

In this paper, we study a time--dependent Allen--Cahn--like system modelling the evolution of the macroscopic strain and fluid density in a porous media which is able to produce steady states exhibiting a strong phase separation between  
fluid--rich and fluid--poor regions; for details see \cite{CIS2010,CIS2011,CIS2013}. 
The system we are studying (referred to as problem $(P)$ in Section \ref{Strong}) has two mathematically challenging components: (i) a coupled flux (a linear combination of strain and fluid density gradients) resembling this way with 
cross--diffusion problems (see \cite{Vanag}, e.g.) or with thermo--diffusion problems (see \cite{NHM}, e.g.); (ii) the polynomial structure of the production term.

Trusting the working techniques from 
\cite{zamm1}, we apply a variant of the Leray--Schauder fixed point theorem to prove the existence of strong solutions to a regularized consolidation problem (see Section \ref{regularized}) and then employ weak convergence methods for this auxiliary problem to obtain  in the limit of the vanishing regularisation parameter local--in--time weak solutions of the original consolidation problem.   Under some additional restrictions on the model parameters, we show that the weak solutions exist globally in time and are negative. We conclude the paper with numerical illustrations of the solution to our problem and point out their non--uniqueness at stationarity for critical parameter regimes. We also briefly discuss a few mathematical aspects still open in this context. 


\section{Problem and results} 
\label{s:pr}
In this Section, we introduce the problem we are interested in and 
state our main results. 
In Section~\ref{s:cons} we shall discuss the our main physical motivations
coming from the porous media physics.

\subsection{Strong formulation of the problem}\label{Strong}
If $\varepsilon$ denotes the strain and $m$ the fluid density of our porous media (say $\Omega$) during a given observation time interval (say $S$), then  
the strong formulation of the problem we are going to study reads as follows:
\begin{align}
&\frac{\partial\varepsilon}{\partial t}+\mbox{div}(-k_1\nabla\varepsilon-k_2\nabla m)=\hat{f}_1(m,\varepsilon)&\mbox{in}\,\Omega\times S,
\label{problemadipartenzaconepsilonpunto}
\\
&\frac{\partial m}{\partial t}+\mbox{div}(-k_2\nabla\varepsilon-k_3\nabla m)=\hat{f}_2(m,\varepsilon)&\mbox{in}\,\Omega\times S,
\label{2problemadipartenzaconepsilonpunto}
\\
&\varepsilon(x,0)=\varepsilon_0(x)&\mbox{in}\,\Omega,
\label{3problemadipartenzaconepsilonpunto}
\\
&m(x,0)=m_0(x)&\mbox{in}\,\Omega, 
\label{4problemadipartenzaconepsilonpunto}
\\
&\varepsilon(l_1,t)=\varepsilon_D(t)&\mbox{in}\,S,
\label{5problemadipartenzaconepsilonpunto}
\\
& m(l_1,t)=m_D(t)&\mbox{in}\,S,
\label{6problemadipartenzaconepsilonpunto}
\\
&\frac{\partial \varepsilon}{\partial x}(l_2,t)=0&\mbox{in}\,S,
\label{7problemadipartenzaconepsilonpunto}
\\
&\frac{\partial m}{\partial x}(l_2,t)=0&\mbox{in}\,S.
\label{8problemadipartenzaconepsilonpunto}
\end{align}
We refer to 
\eqref{problemadipartenzaconepsilonpunto}--\eqref{5problemadipartenzaconepsilonpunto} as problem $(P)$.

This paper targets at the weak solvability of problem $(P)$. 
Before stating our main results, we collect the assumptions imposed on the data and parameters involved in the model equations. 
\begin{itemize}
\item[$H_1$:] The boundary functions $\varepsilon_D(t),\,m_D(t)$ are negative continuous for all $t\in S$ with $|\partial_t\varepsilon_D|,\,|\partial_tm_D|\leq C$ for a positive constant $C$.
\item[$H_2$:] $\varepsilon_0,\,m_0\in C(\overline{\Omega})$ with $\varepsilon_0\leq 0,\,m_0\leq 0$.
\item[$H_3$:] Let $M_1,\,M_2\in\mathbb{R}$ sufficiently large. We take 
\begin{eqnarray}
\hat{f}_1(r,s):=\begin{cases}
f_1(r,s),\,\hfill  \mbox{if}\;|r|\leq M_1\, \mbox{and}\, |s|\leq M_2 \\
0,\hfill\mbox{otherwise},
\end{cases}
\label{f1cappello}
\end{eqnarray}
\begin{eqnarray}
\hat{f}_2(r,s):=\begin{cases}
f_2(r,s),\,\hfill \mbox{if}\;|r|\leq M_1\, \mbox{and}\, |s|\leq M_2 \\
0,\hfill\mbox{otherwise},
\end{cases}
\label{f2cappello}
\end{eqnarray}
where  $f_1,\,f_2:\mathbb{R}\times\mathbb{R}\rightarrow\mathbb{R}$.
For the setting of interest, we consider  $f_1,\,f_2$ defined by
\begin{eqnarray}
f_i=\sum_{k_1=0}^{n^{(i)}_1}\sum_{k_2=0}^{n^{(i)}_2}A_{k_1k_2}^{(i)}\varepsilon^{k_1}m^{k_2},\qquad A_{k_1k_2}^{(i)}\in\mathbb{R},\,n_j^{(i)}\in\mathbb{N},k_i\in\{0,\dots,n_j^{(i)}\},\;
i,j=1,2,
\label{f1f2}
\end{eqnarray}
namely we take $f_1,\,f_2$ to be generic polynomials.


\item[$H_4$:] $k_1,\,k_2,\,k_3$ and $\gamma$ are strictly positive constants.
\item[$H_5$:] $k_2<\mbox{min}\left\lbrace k_1,k_3\right\rbrace$.
\item[$H_6$:] $\varepsilon_0,\,m_0\in C^\nu(\overline{\Omega}),\;\nu>0$.
\end{itemize}
Assumptions $H_1,\,H_2$ reflect the properties of suitably rescaled and translated mechanical strain and fluid density; $H_3,\,H_4$ are made so that this scenario fits to the setting described in 
\cite{CIS2011}, while $H_5,\,H_6$ are technical assumptions.

\subsection{Notation}
\label{sec:notation}
For a function $g=g(x,t),\,\partial_x g\,(\mbox{or}\,\nabla g)$, $\partial_t g$ indicate the partial derivatives with respect to spatial variable $x$ and temporal variable $t$. Let $T,\,l_1\,,l_2>0$ be fixed values. 
Define  $\Omega:=(l_1,l_2)$, $S:=(0,T]$, and 
$L:=|\Omega|=l_2-l_1$. For $1\leq p\leq \infty$ we denote by $L^p(\Omega)$ the usual Lebesgue space equipped with the norm $\|\cdot\|_{L^p(\Omega)}$. For $1\leq p\leq \infty$ and $k$ a positive integer, let $W^{k,p}(\Omega)$ be the usual Sobolev space with the norm $\|\cdot\|_{W^{k,p}(\Omega)}$. We write $H^s(\Omega)$ and $\|\cdot\|_{H^s(\Omega)}$ instead of $W^{s,2}(\Omega)$ and $\|\cdot\|_{W^{s,2}(\Omega)}$. Let $\Gamma^D:=\left\lbrace \ l_1\ \right\rbrace$. We denote by $V$ the space $$V:=\left\lbrace\varphi\in H^1(\Omega),\,\varphi=0\,\mbox{ on }\,\Gamma^D \right\rbrace$$ and recall the equivalence 
\begin{eqnarray}
\|\cdot\|_{H^1(\Omega)}\sim\|\cdot\|_{V(\Omega)}.
\end{eqnarray}
For $p\in[1,\infty)$ we denote by $L^p(S;B)$ the usual Bochner space equipped with the norm
$$
\|\cdot\|_{L^p(S;B)},
$$
for any arbitrary Banach space $B$ equipped with  norm $\|\cdot\|_B$. 
\\\\
Let $B_0$, $B_1$ be two Banach spaces. Define the space
\begin{eqnarray}
W:=\left\lbrace v;\;v\in L^2(S;B_0),\; \partial_t v\in L^2(S;B_1)\right\rbrace,
\label{spazioW}
\end{eqnarray}
and take, as a particular case, $B_0=V,\,B_1=L^2(\Omega)$. By \cite{zeidler}, Proposition 23.23 $(ii)$ p. 422, we have that 
\begin{eqnarray}
W\hookrightarrow\hookrightarrow C(\bar{S};L^2(\Omega)).
\end{eqnarray}
The following compactness result due to \cite{aubin} will be useful in our context:
\begin{thm}[Aubin]\label{aubin}
Let $B_0$, $B$ $B_1$ be three Banach spaces where $B_0$, $B_1$ are reflexive. Suppose that $B_0$ is continuously imbedded into $B$, which is also continuously imbedded into $B_1$, and, morover, the imbedding from $B_0$ into $B$ is compact. Let $W$ be defined as in \eqref{spazioW}.
Then the imbedding from $W$ into $L^2(S;B)$ is compact.  
\end{thm}
Define the space
\[
V_2^{2,1}(\Omega\times S):=\left\lbrace\varphi\in L^2(\Omega\times S),\; \varphi_t,\varphi_x,\varphi_{xx}\in L^2(\Omega\times S)\right\rbrace,
\]
then the following imbedding is a consequence of Theorem \ref{aubin}:
\begin{eqnarray}
V_2^{2,1}(\Omega\times S)\hookrightarrow\hookrightarrow L^2(S;H^1(\Omega)).
\label{imbedding1}
\end{eqnarray}

\subsection{Main results}
\label{sec:main}

\begin{Definition}\label{weaksol}
The couple 
$$(\varepsilon,m)\in \left[\varepsilon_D+L^2(S;V)\cap H^1(S;L^2(\Omega))\right]\times\left[m_D+ L^2(S;V)\cap H^1(S;L^2(\Omega))\right]$$
is called weak solution to problem $(P)$ if and only if the following identities
\begin{eqnarray}
(\partial_t\varepsilon,\varphi)_{L^2(\Omega)}+k_1(\nabla\varepsilon,\nabla\varphi)_{L^2(\Omega)}=-k_2(\nabla m,\nabla\varphi)_{L^2(\Omega)}+(\hat{f}_1,\varphi)_{L^2(\Omega)},
\end{eqnarray}
\begin{eqnarray}
(\partial_t m,\psi)_{L^2(\Omega)}+k_3(\nabla m),\nabla\psi)_{L^2(\Omega)}=-k_2(\nabla\varepsilon,\nabla\psi)_{L^2(\Omega)}+(\hat{f}_2,\psi)_{L^2(\Omega)}.
\end{eqnarray}
hold for all $(\varphi,\psi)\in V\times V$ and for all $t\in S$.
\end{Definition}
\begin{thm}[Existence]\label{main}
Under the assumptions $H_1$--$H_4$ 
there exist at least a weak solution to problem $(P)$ in the sense of Definition \ref{weaksol}.
\end{thm}
\begin{thm}[Uniqueness]\label{uniqueness}
Assume $H_1$--$H_4$ and $H_5$ to hold. Then, for any fixed $T\in(0\,\infty)$, it exist at most a solution to $(P)$ in the sense of Definition \ref{weaksol}.
\end{thm}

\begin{thm}[Boundedness and negativity of $\varepsilon$ and $m$]
Assume $H_1$--$H_4$ and $H_6$ to hold true together with 
$A_1$--$A_3$ (cf.\ Section \ref{sec:negativity}). Moreover, assume that the functions $\hat{f}_1,\,\hat{f}_2$ are negative. Then the solution $(\varepsilon,m)$ to Problem $(P)$ is bounded and negative.
\label{boundandnegative}
\end{thm}

\begin{Remark}
It is worth noting that Theorem \ref{main}, Theroem \ref{uniqueness} and Theorem \ref{boundandnegative} were obtained for the case of Dirchlet--Neumann boundary conditions. Note that with minimal modifications of the proofs, we can handle other kinds of physically relevant boundary conditions (e.g the periodic case or the Dirichlet--Dirichlet or the  Neumann--Neumann boundary conditions).
\end{Remark}

\subsection{Application to the consolidation of porous media}
\label{s:cons}
The problem $(P)$ introduced in Section~\ref{Strong}, as 
already announced in the introduction, has a relevant
application to the theory of Porous Media. In this section we give 
a very brief account of this theory and we refer the interested reader 
to the paper \cite{CIS2013} for a detailed derivation.

We introduce the one dimensional poromechanical model (see 
\cite{CIS2013})
whose geometrically linearized version is connected to problem $(P)$.
Kinematics will be briefly resumed starting from the
general statement of the model cf. \cite{Coussy}.
The equations governing 
the behavior of the porous system are then deduced prescribing 
the conservative part of
the constitutive law through a suitable potential energy density 
$\Phi$ and the dissipative 
contributions through purely Stokes term. 

Let $B_\textrm{s}:=[\ell_1,\ell_2]\subset\mathbb{R}$, with 
$\ell_1,\ell_2\in\mathbb{R}$, 
and $B_\textrm{f}:=\mathbb{R}$ be the \textit{reference}
configurations
for the solid and fluid components; see ~\cite{Coussy}. 
The \textit{solid placement} 
$\chi_\textrm{s}:B_\textrm{s}\times\mathbb{R}\to\mathbb{R}$ is a $C^2$ function such that 
the map $\chi_\textrm{s}(\cdot,t)$, 
associating to each $X_\textrm{s}\in B_\textrm{s}$
the position occupied at time $t$ by the particle labeled 
by $X_\textrm{s}$ in the reference configuration $B_\textrm{s}$,
is a $C^2$--diffeomorphism.
The \textit{fluid placement} map 
$\chi_\textrm{f}:B_\textrm{f}\times\mathbb{R}\to\mathbb{R}$
is defined analogously.
The \textit{current configuration} $B_t:=\chi_\textrm{s}(B_\textrm{s},t)$ at time 
$t$ is the set of positions 
of the superposed solid and fluid particles.
Consider the $C^2$ function
$\phi:B_\textrm{s}\times\mathbb{R}\to B_\textrm{f}$ 
such that 
$\phi(X_\textrm{s},t)$ is 
the fluid particle that at time $t$ occupies the 
same position of the solid particle $X_\textrm{s}$; 
assume, also, that $\phi(\cdot,t)$ is a $C^2$--diffeomorphism 
mapping univocally a solid particle 
into a fluid one.
The three fields
$\chi_\textrm{s}$, $\chi_\textrm{f}$, and $\phi$ are 
not at all independent; indeed, by definition, we immediately have that 
$\chi_\textrm{f}(\phi(X_\textrm{s},t),t)=\chi_\textrm{s}(X_\textrm{s},t)$ 
for any $X_\textrm{s}\in B_\textrm{s}$ and $t\in\mathbb{R}$.

The Lagrangian velocities are two maps associating with each time and
each point in the solid and fluid reference space the velocities of the
corresponding solid and fluid particles at the
specified time.
More precisely,
the \textit{Lagrangian velocities} are the two maps
$u_\alpha:B_\alpha\times\mathbb{R}\to\mathbb{R}$
defined by setting
$u_\alpha(X_\alpha,t):=\partial\chi_\alpha/\partial t$
for any $X_\alpha\in B_\alpha$, where $\alpha=\textrm{s},\textrm{f}$.
We also consider the \textit{Eulerian velocities}
$v_\alpha:B_t\times\mathbb{R}\to\mathbb{R}$ associating
with each point $x\in B_t$ and for each time $t\in\mathbb{R}$ the velocities
of the solid and fluid particle occupying the place $x$ at time $t$;
more precisely we set
$v_\alpha(x,t):=u_\alpha(\chi^{-1}_\alpha(x,t),t)$.

In studying the dynamics of the porous system, we can arbitrarily 
choose two among the three fields 
$\chi_\textrm{s}$, $\chi_\textrm{f}$, and $\phi$. 
Since the reference configuration $B_\textrm{s}$ of the solid component 
is known {\em a priori}, 
a good choice appears to be that of expressing all the dynamical observables 
in terms of the fields $\chi_\textrm{s}$ and $\phi$ which are defined 
on $B_\textrm{s}$. 

It is natural to assume that, if the system is acted upon only by 
conservative forces, its dynamics is described by a 
\textit{Lagrangian density} 
$L$,
relative to the solid reference configuration space volume,
depending on the space variable $X_\textrm{s}$ and on time 
through (in principle)
$\dot\chi_\textrm{s}$, $\dot\phi$, $\chi''_\textrm{s}$, $\phi''$,
$\chi'_\textrm{s}$, $\phi'$, $\chi_\textrm{s}$, and $\phi$.
The Lagrangian density is
equal to the 
\textit{kinetic energy density} minus
the \textit{overall potential energy density} accounting for both 
the internal and the external conservative forces.

Suppose the 
fluid component of the system is acted upon by dissipative
forces.
We consider the independent variations 
$\delta\chi_\textrm{s}$ 
and $\delta\phi$ of the two fields $\chi_\textrm{s}$ and $\phi$ and
denote by $\delta W$ the corresponding elementary 
\textit{virtual work}
made by the dissipative forces acting on the fluid component.
The possible motions of the system, see for instance 
\cite[Chapter~5]{Blanchard},
in an interval of time $(t_1,t_2)\subset\mathbb{R}$
are those such that the fields $\chi_\textrm{s}$ and $\phi$ satisfies
the variational principle
\begin{equation}
\label{vardiss}
\delta 
 \int_{t_1}^{t_2}\textrm{d}t
 \int_{B_\textrm{s}}\textrm{d}X_\textrm{s}
 \,
 {L}(
            \dot\chi_\textrm{s}(X_\textrm{s},t),
            \dots,
            \phi(X_\textrm{s},t))
=
-\int_{t_1}^{t_2}\delta W\,\textrm{d}t
\end{equation}
namely, the variation of the the action integral in 
correspondence of a possible motion is equal to the 
integral over time of minus the virtual work of the dissipative 
forces corresponding to the considered variation of the fields. 

The way in which dissipation has to be 
introduced in saturated porous media 
models is still under debate.
In particular, according to the effectiveness of the hypothesis 
of separation of scales, between the local and macroscopic level,
Darcy's or Stokes' effects are accounted for. We refer 
the interested reader to \cite{CIS2013} for a detailed discussion of this 
issue. 
In this paper, we 
consider the so--called Stokes' effect, i.e., the dissipation 
due to forces controlled by the second derivative 
of the velocity of the 
fluid component measured with respect to the solid.
A natural expression \cite{CIS2013} is 
\begin{equation}
\label{lv01}
 \delta W:=
 -
 \int_{B_t}
 S
 [v_\textrm{f}(x,t)-v_\textrm{s}(x,t)]'
 \,
 [\delta\chi_\textrm{f}(\chi^{-1}_\textrm{f}(x,t),t)
 -
 \delta\chi_\textrm{s}(\chi^{-1}_\textrm{s}(x,t),t)]'
\,\textrm{d}x
\end{equation}
where $\delta\chi_\textrm{f}$ is the variation of the field $\chi_\textrm{f}$
induced by the independent variations $\delta\chi_\textrm{s}$ and 
$\delta\phi$, and $S>0$.

In order to write explicitly the variation of the action one has to 
specify the form of the Lagrangian density.
In the sequel we shall not consider the inertial 
effects, so that, the Lagrangian density will be the opposite of 
the potential energy $\Phi$
density associated to both the internal and external 
conservative forces. 
It is reasonable to assume that the potential energy density 
depends on the space and time variable only via two 
physically relevant functions: 
the strain of the solid and a properly normalized fluid mass 
density~\cite{CIS2013}, i.e.,
\begin{equation}
\label{deformazione}
\varepsilon(X_\textrm{s},t):=[(\chi'_\textrm{s}(X_\textrm{s},t))^2-1]/2
\;\;\textrm{ and }\;\;
m_\textrm{f}(X_\textrm{s},t)
 :=\varrho_{0,\textrm{f}}(\phi(X_\textrm{s},t))
   \phi'(X_\textrm{s},t)
\end{equation}
where $\varrho_{0,\textrm{f}}:B_\textrm{f}\to\mathbb{R}$ 
is a fluid reference \textit{density}.
In other words,  we assume that 
the potential energy density $\Phi$ is a function of the 
fields 
$m_\textrm{f}$ and $\varepsilon$ and on their space
derivative
$m'_\textrm{f}$ and $\varepsilon'$.

By a standard variational computation, see \cite[equation (24)]{CIS2013},
one gets the equation of motion. In this framework, we are interested in the 
geometrically linearized version of such equations: 
we assume $\varrho_{0,\textrm{f}}$ to be constant and 
introduce the \textit{displacement fields} 
$u(X_\textrm{s},t)$ and $w(X_\textrm{s},t)$ by setting
\begin{equation}
\label{gd00}
\chi_\textrm{s}(X_\textrm{s},t)=X_\textrm{s}+u(X_\textrm{s},t)
\;\textrm{ and }\;
\phi(X_\textrm{s},t)=X_\textrm{s}+w(X_\textrm{s},t)
\end{equation}
for any $X_\textrm{s}\in B_\textrm{s}$ and $t\in\mathbb{R}$. We then 
assume that $u$ and $w$ are small, together with 
their space and time derivatives, and write 
\begin{equation}
\label{gd01}
m_\textrm{f}=\varrho_{0,\textrm{f}}(1+w'),\;
m:=m_\textrm{f}-\varrho_{0,\textrm{f}}=\varrho_{0,\textrm{f}}w',\;
\varepsilon\approx u',\;
\end{equation}
where $\approx$ means that all the terms of order larger than one 
have been neglected.
We then write 
the equations of motion 
up to the first order in $u$, $w$, and derivatives:
\begin{equation}
\label{gd02}
 \frac{\partial\Phi}{\partial\varepsilon}
 -
 \Big(
 \frac{\partial\Phi}{\partial\varepsilon'}
 \Big)'
 =0
\;\;\textrm{ and }\;\;
 \frac{\partial\Phi}{\partial m}
      -
      \Big(
      \frac{\partial\Phi}{\partial m'}
      \Big)'
 =
 -\frac{S}{\rho^2_{0,\textrm{f}}}\dot{m}
\end{equation}
with boundary conditions that are compatible with the choices of
Dirichlet and Neumann boundary conditions. 

We specialize the Porous Medium model we are studying by choosing the 
second gradient part of the dimensionless potential energy, that is we assume 
\begin{equation}
\label{sec010}
\Phi(m',\varepsilon',m,\varepsilon)
 :=
 \frac{1}{2}[k_1(\varepsilon')^2+2k_2\varepsilon' m'+k_3(m')^2]
 +
 \Psi(m,\varepsilon)
\end{equation}
with $k_1,k_3>0$, $k_2\in\mathbb{R}$ such that $k_1k_3-k_2^2\ge0$.
These parameters provide energy penalties for 
the formation of interfaces; they have the physical dimensions of
squared lengths and, according with the above mentioned 
conditions, provide a well--grounded identification of the intrinsic
characteristic lengths of the one--dimensional porous continuum.
In this case, equations \eqref{gd02} become
\begin{equation}
\label{gd03}
 \frac{\partial\Psi}{\partial\varepsilon}
 -
 (k_1\varepsilon'+k_2 m')'
 =0
\;\;\textrm{ and }\;\;
 \frac{\partial\Psi}{\partial m}
      -
      (k_2\varepsilon'+k_3m')'
 =
 -\frac{S}{\rho^2_{0,\textrm{f}}}\dot{m}.
\end{equation}
We notice immediately that such a system of PDE has the 
form of the problem $(P)$ introduced in Section~\ref{Strong}
provided the first gradient energy $\Psi$ is a polynomial 
in the strain and in the fluid content. 
The main difference between the two system of equations lies in the fact that 
in the first of the two equations the derivative of the strain 
with respect to time is missing. This is due to the fact that, for simplicity
and for coherence with the previous paper on which our sketch of derivation 
is based, 
we have not considered the 
dissipation forces acting on the solid components. 
If those forces would be taken into account, we would get a 
parabolic--parabolic system as the one in problem $(P)$.

An important application of the theory briefly recalled in this section is 
that to the study of phase transitions in porous media under 
consolidation, namely, when the system is acted upon by an 
external pressure. This issue will be discussed in Section~\ref{s:num}.

\section{Proof of Theorem \ref{main}}
\label{s:dim00}

In this Section, we prove Theorem \ref{main} via a Leray-Schauder fixed point argument. Firstly, we study a regularised version of Problem $(P)$ for which we prove the existence of a strong solution, see Theorem \ref{strongsolutionpdelta}. The proof of Theorem \ref{strongsolutionpdelta} is divided in two main steps: In the first one, we introduce an auxiliary problem, depending on a parameter $\zeta\in[0,1]$, for which a direct application of the theory of quasi-linear parabolic equations gives a unique classical  solution. The second step is concerned with  the definition of a 
nonlinear mapping that satisfies the hypothesis of the Leray Schauder argument, Theorem 
\ref{lerayschauder}. Once Theorem \ref{strongsolutionpdelta} is proven, we exploit weak convergence methods to get the obtain the conclusion of the main Theorem \ref{main}.

\subsection{The regularized problem}
\label{regularized}
Let us introduce the following mollified version of problem $(P)$, namely: Find the pair $(\varepsilon,m)$ satisfying
\begin{align}
&\frac{\partial\varepsilon}{\partial t}+\mbox{div}(-k_1\nabla\varepsilon-k_2\nabla^\delta m)=\hat{f}_1(m,\varepsilon)&\mbox{in}\,\Omega,
\label{molliproblemadipartenzaconepsilonpunto}
\\
&\frac{\partial m}{\partial t}+\mbox{div}(-k_2\nabla^\delta\varepsilon-k_3\nabla m)=\hat{f}_2(m,\varepsilon)&\mbox{in}\,\Omega,
\label{2molliproblemadipartenzaconepsilonpunto}
\\
&\varepsilon(0)=\varepsilon_0,&\mbox{in}\,\Omega,
\label{3molliproblemadipartenzaconepsilonpunto}
\\
&m(0)=m_0&\mbox{in}\,\Omega,
\label{4molliproblemadipartenzaconepsilonpunto}
\\
&\varepsilon(l_1,t)=\varepsilon_D(t)&\mbox{in}\,S,
\label{5molliproblemadipartenzaconepsilonpunto}
\\
&m(l_1,t)=m_D(t)&\mbox{in}\,S,
\label{6molliproblemadipartenzaconepsilonpunto}
\\
&\frac{\partial \varepsilon}{\partial x}(l_2,t)=0&\mbox{in}\,S,
\label{7molliproblemadipartenzaconepsilonpunto}
\\
&\frac{\partial m}{\partial x}(l_2,t)=0&\mbox{in}\,S.
\label{8molliproblemadipartenzaconepsilonpunto}
\end{align}

We refer to 
\eqref{molliproblemadipartenzaconepsilonpunto}--\eqref{8molliproblemadipartenzaconepsilonpunto} as problem $(P^\delta)$. In this section, we prove the existence of strong solutions to problem $(P^\delta)$.

To define problem  $(P^\delta)$, we use the following definition of the mollified gradient of a function $f$ [see e.g. \cite{krehel}]:
\begin{eqnarray}
\nabla^\delta f:=\nabla\left[\int_{B_\delta(x)}J_\delta(x-y)f(y) dy\right],
\label{defnabladelta}
\end{eqnarray}
where $J_\delta$ denotes the standard mollifier defined for example in \cite{fournier} and $B_\delta(x)$ is a ball centred in $x\in \Omega$ with radius $\delta>0$ chosen such that $x+\delta\in\Omega$. A mollified function $u$ enjoys of the following properties:
\begin{thm} 
Let $u$ be a function which is defined on $\mathbb{R}^n$ and vanishes identically outside $\Omega$. If $u\in L^p(\Omega)$, $1\leq p<\infty$, then $J_\delta*u\in L^p(\Omega)$. Also
\[
\|J_\delta*u\|_{L^p(\Omega)}\leq\|u\|_{L^p(\Omega)}\qquad \mbox{and} \qquad \lim_{\delta\rightarrow 0^+}\|J_\delta*u-u\|_{L^p(\Omega)}=0.
\]
\label{propmoll}
Also,  for all $f\in L^\infty(\Omega)$ and $1\leq p\leq\infty$,  it exists a constant $c^\delta>0$ such that
\begin{eqnarray}
\|\nabla^\delta f\|_{L^p(\Omega)}\leq c^\delta\|f\|_{L^2(\Omega)}.
\label{propnabladelta}
\end{eqnarray}
\end{thm}
Note that as $\delta\to 0$, typically $c^\delta \to \infty$.

\begin{thm}\label{strongsolutionpdelta}
Assume $H_1$--$H_4$. Problem $(P^\delta)$ has at least a strong solution $$(\varepsilon_\delta,m_\delta)\in\varepsilon_D+ V^{2,1}_2(\Omega\times S)\times m_D+V^{2,1}_2(\Omega\times S).$$
\end{thm}
\begin{proof}
The strong solution of the regularized problem $(P^\delta)$ is obtained here by a direct application of the Leray--Schauder fixed point theorem, viz. 
\begin{thm}[Leray--Schauder Fixed Point Theorem]\label{lerayschauder}
Let $X$ be a Banach space and let $\mathcal{T}$ be a completely continuous mapping of $X\times[0,1]$ into $X$ such that $\mathcal{T}(v,0)=0$ for all $v\in X$. Suppose there exists a constant $R>0$ such that
\begin{eqnarray}
\|v\|_X\leq R
\end{eqnarray} 
for all $(v,\zeta)\in X\times[0,1]$ satisfying $v=\mathcal{T}(v,\zeta)$. Then the mapping $\mathcal{T}_1$ of $X$ into itslef given by $$\mathcal{T}_1(v):=\mathcal{T}(v,1)$$ has a fixed point.
\end{thm}
A nice proof of the Leray--Schauder Theorem can be found e.g. in \cite{gilbarg}, Theorem 11.6; see also \cite{bookprecup}.\\\\
\textbf{Solution to an auxiliary problem.}
For any given couple $(\tilde{\varepsilon},\tilde{m})$ with $$\tilde{\varepsilon}\in \zeta\varepsilon_D+L^2(S;V),\,m\in\zeta m_D+L^2(S;V)$$ and $\zeta \in[0,1]$, consider the initial boundary value problem 
\begin{align}
&\frac{\partial\varepsilon}{\partial t}+\mbox{div}(-k_1\nabla\varepsilon-k_2\nabla^\delta \tilde{m})=\hat{F}_1(\tilde{\varepsilon},\tilde{m},\varepsilon)&\mbox{in}\,\Omega\times S,
\label{auxzetamolliproblemadipartenzaconepsilonpunto}
\\&
\frac{\partial m}{\partial t}+\mbox{div}(-k_2\nabla^\delta\tilde{\varepsilon}-k_3\nabla m)=\hat{F}_2(\tilde{\varepsilon} ,\tilde{m},m)&\mbox{in}\,\Omega\times S,
\label{auxzeta2molliproblemadipartenzaconepsilonpunto}
\\
&\varepsilon(0)=\zeta\varepsilon_0,&\mbox{in}\,\Omega,
\label{auxzeta3molliproblemadipartenzaconepsilonpunto}
\\
& m(0)=\zeta m_0&\mbox{in}\,\Omega,
\label{auxzeta4molliproblemadipartenzaconepsilonpunto}
\\
&\varepsilon(l_1,t)=\zeta\varepsilon_D(t)&\mbox{in}\,S,
\label{auxzeta5molliproblemadipartenzaconepsilonpunto}
\\& m(l_1,t)=\zeta m_D(t)&\mbox{in}\,S,
\label{auxzeta6molliproblemadipartenzaconepsilonpunto}
\\
&\frac{\partial \varepsilon}{\partial x}(l_2,t)=0&\mbox{in}\,S,
\label{auxzeta7molliproblemadipartenzaconepsilonpunto}
\\
&\frac{\partial m}{\partial x}(l_2,t)=0&\mbox{in}\,S.
\label{auxzeta8molliproblemadipartenzaconepsilonpunto}
\end{align}
where we have set 
\begin{equation}
\hat{F}_1(\tilde{\varepsilon},\tilde{m},\varepsilon):=
\sum_{k_2=0}^{n_2^{(1)}}A_{0k_2}^{(1)}\tilde{m}^{k_2}+\sum_{k_2=0}^{n_2^{(1)}}A_{1k_2}^{(1)}\tilde{m}^{k_2}\varepsilon+\sum_{k_1=2}^{n^{(1)}_1}\sum_{k_2=0}^{n^{(1)}_2}A_{k_1k_2}^{(1)}\tilde{m}^{k_2}\tilde{\varepsilon}^{k_1-1}\varepsilon
\end{equation}
if $|\varepsilon|,|\tilde{\varepsilon}|\leq M_1$ and 
$|m|,|\tilde{m}|\leq M_2$ and 
$\hat{F}_1(\tilde{\varepsilon},\tilde{m},\varepsilon):=0$ otherwise, 
and 
\begin{equation}
\hat{F}_2(\tilde{\varepsilon} ,\tilde{m},m):=
\sum_{k_1=0}^{n_1^{(2)}}A_{k_10}^{(2)}\tilde{\varepsilon}^{k_2}+\sum_{k_1=0}^{n_1^{(2)}}A_{k_11}^{(2)}\tilde{\varepsilon}^{k_1}m+\sum_{k_2=2}^{n^{(2)}_2}\sum_{k_1=0}^{n^{(2)}_1}A_{k_1k_2}^{(2)}\tilde{\varepsilon}^{k_1}\tilde{m}^{k_2-1}m
\end{equation}
if $|\varepsilon|,|\tilde{\varepsilon}|\leq M_1$ and 
$|m|,|\tilde{m}|\leq M_2$, and 
$\hat{F}_2(\tilde{\varepsilon} ,\tilde{m},m):= 0$ otherwise.

We note that $\hat{F}_1(\varepsilon,m,\varepsilon)=\hat{f}_1(\varepsilon,m)$ and $\hat{F}_2(\varepsilon,m,m)=\hat{f}_2(\varepsilon,m)$.
We split the proof of existence of solutions to the system \eqref{auxzetamolliproblemadipartenzaconepsilonpunto}--\eqref{auxzeta8molliproblemadipartenzaconepsilonpunto}  into two steps:

\medskip
\par\noindent
\textbf{Step 1.\/} 
The system of equations \eqref{auxzetamolliproblemadipartenzaconepsilonpunto}, \eqref{auxzeta3molliproblemadipartenzaconepsilonpunto}, \eqref{auxzeta5molliproblemadipartenzaconepsilonpunto} and \eqref{auxzeta7molliproblemadipartenzaconepsilonpunto} is a special case of the problem 
\begin{align}
&\varepsilon_t-g_1(x,t)\varepsilon_{xx}-g_2(x,t,\varepsilon,\varepsilon_x)=0&\mbox{in}\,\Omega\times S,\\
&\varepsilon(x,0)=\zeta\varepsilon_0&\mbox{in}\,\Omega,
\\
&\varepsilon(l_1,t)=\zeta\varepsilon_D(t)&\mbox{in}\,S,
\\
&\frac{\partial \varepsilon}{\partial x}(l_2,t)=0 &\mbox{in}\,S.
\end{align}
In our case, we have
\begin{eqnarray*}
&&g_1(x,t)=k_1,\\
&&g_2(x,t,\varepsilon,\varepsilon_x)=-k_2\mbox{div}(\nabla^\delta\tilde{m})+\sum_{k_2=0}^{n_2^{(1)}}A_{0k_2}^{(1)}\tilde{m}^{k_2}
+\sum_{k_2=0}^{n_2^{(1)}}A_{1k_2}^{(1)}\tilde{m}^{k_2}\varepsilon+
\sum_{k_1=2}^{n^{(1)}_1}\sum_{k_2=0}^{n^{(1)}_2}A_{k_1k_2}^{(1)}\tilde{m}^{k_2}\tilde{\varepsilon}^{k_1-1}\varepsilon.
\end{eqnarray*}
Under the assumptions $H_1-H_4$ and trusting the classical theory of quasi-linear parabolic equations (see Theorem 7.4, Chapter V in \cite{lady}), for any given $$(\tilde{\varepsilon},\tilde{m})\in L^2(S;H^1(\Omega))\times L^2(S;H^1(\Omega))$$ the problem \eqref{auxzetamolliproblemadipartenzaconepsilonpunto},\eqref{auxzeta3molliproblemadipartenzaconepsilonpunto},\eqref{auxzeta5molliproblemadipartenzaconepsilonpunto} and \eqref{auxzeta7molliproblemadipartenzaconepsilonpunto} admits the unique solution $\varepsilon\in V_2^{2,1}(\Omega\times S)$.

\medskip
\par\noindent
\textbf{Step 2.\/} The system of equations \eqref{auxzeta2molliproblemadipartenzaconepsilonpunto},\eqref{auxzeta4molliproblemadipartenzaconepsilonpunto},\eqref{auxzeta6molliproblemadipartenzaconepsilonpunto} and \eqref{auxzeta8molliproblemadipartenzaconepsilonpunto} can be treated in an analogous way as in \textbf{Step 1}, in fact such system is a special case of the problem
\begin{align}
&m_t-g_3(x,t)m_{xx}-g_4(x,t,m,m_x)=0&\mbox{in}\,\Omega\times S,\\
&m(x,0)=\zeta m_0&\mbox{in}\,\Omega,
\\
&m(l_1,t)=\zeta m_D(t)&\mbox{in}\,S,
\\
&\frac{\partial m}{\partial x}(l_2,t)=0 &\mbox{in}\,S.
\end{align}
In this case, we have
\begin{eqnarray*}
&&g_3(x,t)=k_3,\\
&&g_4(x,t,m,m_x)=-k_2\mbox{div}(\nabla^\delta\tilde{\varepsilon})-\alpha \tilde{m}^2m-2\alpha b\tilde{\varepsilon}\tilde{m}m-b^2\alpha\tilde{\varepsilon
}^2m+\sum_{k_1=0}^{n_1^{(2)}}A_{k_10}^{(2)}\tilde{\varepsilon}^{k_2}
\\
&&+\sum_{k_1=0}^{n_1^{(2)}}A_{k_11}^{(2)}\tilde{\varepsilon}^{k_1}m
+\sum_{k_2=2}^{n^{(2)}_2}\sum_{k_1=0}^{n^{(2)}_1}A_{k_1k_2}^{(2)}\tilde{\varepsilon}^{k_1}\tilde{m}^{k_2-1}m.
\end{eqnarray*}
Based on the assumptions $H_1$--$H_4$ and relying once more on the theory for quasi--linear parabolic equations (see  Theorem 7.4, Chapter V in \cite{lady}), for any given $$(\tilde{\varepsilon},\tilde{m})\in L^2(S;H^1(\Omega))\times L^2(S;H^1(\Omega))$$ the problem \eqref{auxzeta2molliproblemadipartenzaconepsilonpunto},\eqref{auxzeta4molliproblemadipartenzaconepsilonpunto},\eqref{auxzeta6molliproblemadipartenzaconepsilonpunto} and \eqref{auxzeta8molliproblemadipartenzaconepsilonpunto} admits the unique solution $m\in V_2^{2,1}(\Omega\times S)$.\\\\
Finally, we conclude that for any given couple $(\tilde{\varepsilon},\tilde{m})\in L^2(S;H^1(\Omega))\times L^2(S;H^1(\Omega))$  and $0\leq\zeta\leq 1$, we have $(\varepsilon,m)\in V_2^{2,1}(\Omega\times S)\times V_2^{2,1}(\Omega\times S)$ as a solution to the problem \eqref{auxzetamolliproblemadipartenzaconepsilonpunto}--\eqref{auxzeta8molliproblemadipartenzaconepsilonpunto}.

\begin{Definition}
Denote by $X:=L^2(S;H^1(\Omega))\times L^2(S;H^1(\Omega))$, take arbitrary  $(\tilde{\varepsilon},\tilde{m})\in X,\,\zeta\in[0,1]$ and take the couple $(\varepsilon,m)\in V_2^{2,1}(\Omega\times S)\subset X$ as the solution to problem \eqref{auxzetamolliproblemadipartenzaconepsilonpunto}--\eqref{auxzeta8molliproblemadipartenzaconepsilonpunto}. We define the nonlinear mapping $\mathcal{G}:X\times[0,1]\rightarrow X$ by means of the equation 
\begin{eqnarray}
(\varepsilon,m)=\mathcal{G}(\tilde{\varepsilon},\tilde{m},\zeta).
\end{eqnarray}
\label{defmapping}
\end{Definition}

\noindent\textbf{Basic \textit{a priori} estimates.}
Now, we prove  unifrom estimates for all $(\varepsilon,m)\in X$ satisfying the equation $(\varepsilon,m)=\mathcal{G}(\varepsilon,m,\zeta)$ for some $\zeta\in[0,1]$. In fact, note that if $(\varepsilon,m)\in X$ is a fixed point of $\mathcal{G}(\cdot,\cdot,\zeta),$ then $(\varepsilon,m)\in V_2^{2,1}(\Omega\times S)$. Define now
\[
\hat{G}_1(\tilde{\varepsilon},\tilde{m},\varepsilon):=\hat{F}_1(\tilde{\varepsilon},\tilde{m},\varepsilon)+\partial_t\varepsilon_D,\;\hat{G}_2(\tilde{\varepsilon},\tilde{m},m)_2:=\hat{F}_2(\tilde{\varepsilon},\tilde{m},m)+\partial_tm_D,
\]
and test \eqref{auxzetamolliproblemadipartenzaconepsilonpunto} by $\varphi\in V$ and \eqref{auxzeta2molliproblemadipartenzaconepsilonpunto} by $\psi\in V$.\\ 
We multiply \eqref{auxzetamolliproblemadipartenzaconepsilonpunto} by $\varphi=\varepsilon$ and integrate it over $\Omega\times S$, to get
\begin{eqnarray}
&&\frac{1}{2}\frac{d}{dt}\int_S\|\varepsilon\|^2_{L^2(\Omega)}ds+k_1\int_S\|\nabla\varepsilon\|^2_{L^2(\Omega)}ds=
\nonumber\\&&-k_2\int_S(\nabla^\delta\tilde{m},\nabla\varepsilon)_{L^2(\Omega)}ds+\int_S(\hat{G}_1(\tilde{\varepsilon},\tilde{m},\varepsilon),\varepsilon)_{L^2(\Omega)}ds.
\label{newestimateepsilon1}
\end{eqnarray}
Consequently, we obtain
\begin{eqnarray}
&&\frac{1}{2}\frac{d}{dt}\int_S\|\varepsilon\|^2_{L^2(\Omega)}ds+k_1\int_S\|\nabla\varepsilon\|^2_{L^2(\Omega)}ds\\\nonumber
&&\leq \int_S\left|(\hat{G}_1(\tilde{\varepsilon},\tilde{m},\varepsilon),\varepsilon)_{L^2(\Omega)}\right|ds
+c_\eta k_2^2\int_S\|\nabla^\delta\tilde{m}\|^2_{L^2(\Omega)}ds+\eta\int_S\|\nabla\varepsilon\|^2_{L^2(\Omega)}ds,
\end{eqnarray}
where we have applied the Young inequality on the right hand side of \eqref{newestimateepsilon1}. It is easy to prove that there exist $c_1\,,c_3>0$ such that
\begin{eqnarray}
\int_S\left|(\hat{G}_1(\tilde{\varepsilon},\tilde{m},\varepsilon),\varepsilon)_{L^2(\Omega)}\right|ds\leq c_1+c_3\int_S\|\varepsilon\|^2_{L^2(\Omega)}ds.
\end{eqnarray}
Thus, defining $C_1:=c_1+c_\eta k_2^2T\|\nabla^\delta\tilde{m}\|^2_{L^\infty(\Omega)},\,c_2:=(k_1-\eta)$, and then taking $\eta<k_1$, we obtain
\begin{eqnarray}
\frac{1}{2}\left(\|\varepsilon\|^2_{L^2(\Omega)}-\|\varepsilon(0)\|^2_{L^2(\Omega)}\right) +c_2\int_S\|\nabla\varepsilon\|^2_{L^2(\Omega)}ds\leq C_1+c_3\int_S\|\varepsilon\|^2_{L^2(\Omega)}ds.
\end{eqnarray}
Applying Gronwall's inequality, we get the following uniform estimate for $\varepsilon$
\begin{eqnarray}
\|\varepsilon\|_{L^2(S;V)}\leq C,
\label{energyestimateepsilon}
\end{eqnarray}
for a positive constant $C$ independent of $\zeta$ and $\delta$.\\
We multiply now \eqref{auxzeta2molliproblemadipartenzaconepsilonpunto} by $\psi=m$ and then integrate it over $\Omega\times S$, to get
\begin{eqnarray}
&&\frac{1}{2}\frac{d}{dt}\int_S\|m\|^2_{L^2(\Omega)}ds+k_3\int_S\|\nabla m\|^2_{L^2(\Omega)}ds=
\\\nonumber&&-k_2\int_S(\nabla^\delta\tilde{\varepsilon},\nabla m)_{L^2(\Omega)}ds+\int_S(\hat{G}_2(\tilde{\varepsilon},\tilde{m},m),m)_{L^2(\Omega)}ds.
\label{newestimatem1}
\end{eqnarray}
Consequently,
\begin{eqnarray}
&&\frac{1}{2}\frac{d}{dt}\int_S\|m\|^2_{L^2(\Omega)}ds+k_3\int_S\|\nabla m\|^2_{L^2(\Omega)}ds
\leq
\nonumber\\&&\int_S(\hat{G}_2(\tilde{\varepsilon},\tilde{m},m),m)_{L^2(\Omega)}ds
+c_\eta k_2^2\int_S\|\nabla^\delta\tilde{\varepsilon}\|_{L^2(\Omega)}ds+\eta\int_S\|\nabla m\|^2_{L^2(\Omega)}ds,
\end{eqnarray}
where we have applied the Young inequality on the right hand side of \eqref{newestimatem1}. Now it is easy to find $c_4>0$ and $c_6>0$ such that
\begin{eqnarray}
\int_S\left|(\hat{G}_2(\tilde{\varepsilon},\tilde{m},m),m)_{L^2(\Omega)}\right|ds\leq c_4+c_6\int_S\|m\|^2_{L^2(\Omega)}ds.
\end{eqnarray}
Take $C_4:=c_4+c_\eta k_2^2T\|\nabla^\delta\tilde{\varepsilon}\|^2_{L^\infty(\Omega)},\,c_5:=(k_3-\eta)$, and then taking $\eta<k_3$, we obtain
\begin{eqnarray}
\frac{1}{2}\left(\|m\|^2_{L^2(\Omega)}-\|m(0)\|^2_{L^2(\Omega)}\right)+c_5\int_S\|\nabla m\|^2_{L^2(\Omega)}ds\leq C_4+c_6\int_S\|m\|^2_{L^2(\Omega)}ds.
\end{eqnarray}
By the Gronwall argument, we get the desired uniform estimate for $m$
\begin{eqnarray}
\|m\|_{L^2(S;V)}\leq C,
\label{energyestimatem}
\end{eqnarray}
for a positive constant $C$ independent of $\zeta$ and $\delta$.\\ 
The following estimates are a direct consequence of \eqref{energyestimateepsilon} and \eqref{energyestimatem}:
\begin{eqnarray}
\label{energyestimateinftyepsilon1}
\|\varepsilon\|_{L^\infty(S;L^2(\Omega))}\leq c,\\
\|m\|_{L^\infty(S;L^2(\Omega))}\leq c.
\label{energyestimateinftym1}
\end{eqnarray}

Now, we observe that, by construction, actually $\hat{f}_1\in L^\infty(S;L^\infty(\Omega))$. In particular $\hat{f}_1\in L^2(S;L^2(\Omega))$, so we can show that
\begin{eqnarray}
\|\partial_t\varepsilon\|_{L^2(S;L^2(\Omega))}\leq C.
\label{stimaepsilot2}
\end{eqnarray}
The same property holds for $m$, i. e.
\begin{eqnarray}
\|\partial_t m\|_{L^2(S;L^2(\Omega))}\leq C.
\label{stimamt2}
\end{eqnarray}
\\\\Having established these basic estimates we are ready to complete the proof of Theorem \ref{strongsolutionpdelta} using the Leray--Schauder approach.
\\\\
\textbf{Leray--Schauder fixed point argument.} Take $X:=L^2(S;V)\times L^2(S;V)$  and $\mathcal{G}:X\times[0,1]\rightarrow X$ defined by $\mathcal{G}(\tilde{\varepsilon},\tilde{m},\zeta)=(\varepsilon,m)$ (see Definition \ref{defmapping}), where $(\varepsilon,m)$ is the solution to the auxiliary problem \eqref{auxzetamolliproblemadipartenzaconepsilonpunto}--\eqref{auxzeta8molliproblemadipartenzaconepsilonpunto}. 

First of all, let us prove that $\mathcal{G}:X\times[0,1]\rightarrow X$ is continuous. To this aim we follow the spirit of \cite{satta}. For any sequence $(\tilde{\varepsilon}_n,\tilde{m}_n,\zeta_n)\in X\times[0,1]$ such that
$$(\tilde{\varepsilon}_n,\tilde{m}_n,\zeta_n)\rightarrow(\varepsilon,m,\zeta)\quad\mbox{in}\,X\times[0,1],$$
we denote by $\varepsilon_n$ the solution of the auxiliary problem
\begin{align}
&\partial_t\varepsilon_n-g_1(x,t)\varepsilon_{n,xx}-g_2(x,t,\varepsilon_n,\varepsilon_{n,x})=0&\mbox{in}\,\Omega\times S,\\
&\varepsilon_n(x,0)=\zeta_n\varepsilon_0&\mbox{in}\,\Omega,
\\
&\varepsilon_n(l_1,t)=\zeta_n\varepsilon_D(t)&\mbox{in}\,S,
\\
&\frac{\partial \varepsilon_n}{\partial x}(l_2,t)=0 &\mbox{in}\,S,
\end{align}
and by $m_n$ the solution of the auxiliary problem
\begin{align}
&\partial_tm_n-g_3(x,t)m_{xx}-g_4(x,t,m,m_x)=0&\mbox{in}\,\Omega\times S,\\
&m_n(x,0)=\zeta_nm_0&\mbox{in}\,\Omega,
\\
&m_n(l_1,t)=\zeta_n m_D(t)&\mbox{in}\,S,
\\
&\frac{\partial m_n}{\partial x}(l_2,t)=0 &\mbox{in}\,S,
\end{align}
while $(\varepsilon,m)$ is the solution to the auxiliary problem \eqref{auxzetamolliproblemadipartenzaconepsilonpunto}--\eqref{auxzeta8molliproblemadipartenzaconepsilonpunto}.\\Subtracting the corresponding equations and testing with $\varepsilon_n-\varepsilon$ and $m_n-m$, we obtain:
\begin{eqnarray}
&&(\partial_t\varepsilon_n-\partial_t\varepsilon,\varepsilon_n-\varepsilon)_{L^2(\Omega)}+(\mbox{div}(-k_1\nabla\varepsilon_n+k_1\nabla\varepsilon),\varepsilon_n-\varepsilon)_{L^2(\Omega)}
\nonumber\\&&+(\mbox{div}(-k_2\nabla^\delta\tilde{m}_n+k_2\nabla^\delta\tilde{m}),\varepsilon_n-\varepsilon)_{L^2(\Omega)}=(\hat{F}_1(\tilde{\varepsilon}_n,\tilde{m}_n,\varepsilon)-\hat{F}_1(\tilde{\varepsilon},\tilde{m},\varepsilon),\varepsilon_n-\varepsilon)_{L^2(\Omega)},
\nonumber\\
\label{1continuitymapping}
\end{eqnarray}
and
\begin{eqnarray}
&&(\partial_tm_n-\partial_tm,m_n-m)_{L^2(\Omega)}+(\mbox{div}(-k_3\nabla m_n+k_3\nabla m),m_n-m)_{L^2(\Omega)}
\nonumber\\&&+(\mbox{div}(-k_2\nabla^\delta\tilde{\varepsilon}_n+k_2\nabla^\delta\tilde{\varepsilon}),m_n-m)_{L^2(\Omega)}=(\hat{F}_2(\tilde{\varepsilon}_n,\tilde{m}_n,m)-\hat{F}_2(\tilde{\varepsilon},\tilde{m},m),m_n-m)_{L^2(\Omega)}.
\nonumber\\
\label{2continuitymapping}
\end{eqnarray}
Using Young's inequality in \eqref{1continuitymapping}, we get for $\eta>0$
\begin{eqnarray}
&&\frac{1}{2}\frac{d}{dt}\|\varepsilon_n-\varepsilon\|^2_{L^2(\Omega)}+k_1\|\nabla(\varepsilon_n-\varepsilon)\|^2_{L^2(\Omega)}\leq \eta\|\nabla(\varepsilon_n-\varepsilon)\|^2_{L^2(\Omega)}
\nonumber\\&&+c_\eta k_2^2\|\nabla^\delta(\tilde{m}_n-\tilde{m})\|^2_{L^2(\Omega)}
+(\hat{F}_1(\tilde{\varepsilon}_n,\tilde{m}_n,\varepsilon)-\hat{F}_1(\tilde{\varepsilon},\tilde{m},\varepsilon),\varepsilon_n-\varepsilon)_{L^2(\Omega)}.
\label{3continuitymapping}
\end{eqnarray}
Now, it is straightforward to show that
\begin{eqnarray}
\left|\hat{F}_1(\tilde{\varepsilon}_n,\tilde{m}_n,\varepsilon)-\hat{F}_1(\tilde{\varepsilon},\tilde{m},\varepsilon)\right|\leq c_1|\tilde{\varepsilon
}_n-\tilde{\varepsilon}|+c_2|\tilde{m}_n-\tilde{m}|. 
\end{eqnarray}
The constants $c_1$ and $c_2$ can be computed here explicitly if needed.

From \eqref{2continuitymapping} and proceeding in analogous way as for \eqref{1continuitymapping}, we have for $\eta>0$
\begin{eqnarray}
&&\frac{1}{2}\frac{d}{dt}\|m_n-m\|^2_{L^2(\Omega)}+k_3\|\nabla(m_n-m)\|^2_{L^2(\Omega)}\leq \eta\|\nabla(m_n-m)\|^2_{L^2(\Omega)}
\nonumber\\&&+c_\eta k_2^2\|\nabla^\delta(\tilde{\varepsilon}_n-\tilde{\varepsilon})\|^2_{L^2(\Omega)}+
(\hat{F}_2(\tilde{\varepsilon}_n,\tilde{m}_n,m)-\hat{F}_2(\tilde{\varepsilon},\tilde{m},m),m_n-m)_{L^2(\Omega)}.
\label{4continuitymapping}
\end{eqnarray}
It holds, as before, that 
\begin{eqnarray}
\left|\hat{F}_2(\tilde{\varepsilon}_n,\tilde{m}_n,m)-\hat{F}_2(\tilde{\varepsilon},\tilde{m},m)\right|\leq c_3|\tilde{\varepsilon}_n-\tilde{\varepsilon}|+c_4|\tilde{m}_n-\tilde{m}|.
\end{eqnarray}
Summing up \eqref{3continuitymapping} and \eqref{4continuitymapping} and using  Young's  inequality and the property \eqref{propnabladelta} in combination  with Poincar\'{e}'s inequality, we obtain
\begin{eqnarray}
&&\frac{1}{2}\frac{d}{dt}\left(\|\varepsilon_n-\varepsilon\|^2_{L^2(\Omega)}+
\|m_n-m\|^2_{L^2(\Omega)}\right)
+(k_1-\eta)\|\nabla(\varepsilon_n-\varepsilon)\|^2_{L^2(\Omega)}+(k_3-\eta)\|\nabla(m_n-m)\|^2_{L^2(\Omega)}
\nonumber\\&&
\leq C_1\|\nabla(\tilde{m}_n-\tilde{m})\|^2_{L^2(\Omega)}
+C_2\|\nabla(\tilde{\varepsilon}_n-\tilde{\varepsilon})\|^2_{L^2(\Omega)}+c_\eta c_3\|\tilde{\varepsilon}_n-\tilde{\varepsilon}\|_{L^2(\Omega)}^2+c_\eta(c_2+c_4)\|\tilde{m}_n-\tilde{m}\|_{L^2(\Omega)}^2
\nonumber\\&&+\eta\int_\Omega(\varepsilon_n-\varepsilon)^2dx+\eta\int_\Omega(m_n-m)^2dx.
\end{eqnarray}
We make use again of the Poincar\'{e} inequality, so that we finally get
\begin{eqnarray}
&&\frac{1}{2}\frac{d}{dt}\left(\|\varepsilon_n-\varepsilon\|^2_{L^2(\Omega)}+
\|m_n-m\|^2_{L^2(\Omega)}\right)
+C_3\|\nabla(\varepsilon_n-\varepsilon)\|^2_{L^2(\Omega)}+C_4\|\nabla(m_n-m)\|^2_{L^2(\Omega)}
\nonumber\\&&
\leq C_1\|\nabla(\tilde{m}_n-\tilde{m})\|^2_{L^2(\Omega)}
+C_2\|\nabla(\tilde{\varepsilon}_n-\tilde{\varepsilon})\|^2_{L^2(\Omega)}+\eta c_3\|\tilde{\varepsilon}_n-\tilde{\varepsilon}\|_{L^2(\Omega)}^2+\eta(c_2+c_4)\|\tilde{m}_n-\tilde{m}\|_{L^2(\Omega)}^2.
\nonumber\\
\label{quasifinalecontinuity}
\end{eqnarray}
where $C_3:=k_1-\eta-\eta Cc_3,\,C_4:=k_3-\eta-C\eta(c_2+c_4)$ and $C$ is the costant of the Poincar\'{e} inequality. From \eqref{quasifinalecontinuity}, we obtain (in a compact form)
\begin{eqnarray}
\|(\varepsilon_n,m_n,\zeta_n)-(\varepsilon,m,\zeta)\|^2_X\leq c\left(\|(\tilde{\varepsilon}_n,\tilde{m}_n,\tilde{\zeta}_n)-(\tilde{\varepsilon},\tilde{m},\tilde{\zeta})\|_X^2\right).
\end{eqnarray}
Thus the continuity of $\mathcal{G}$ is proven.\\

Let us now prove that $\mathcal{G}:X\times[0,1]\rightarrow X$ is compact. By the estimates \eqref{energyestimateepsilon} and \eqref{energyestimatem}, $\mathcal{G}$ maps bounded sets from $X\times[0,1]$ into bounded sets of $V_2^{2,1}(\Omega\times S)$. Since the embedding $V_2^{2,1}(\Omega\times S)$ into $L^2(S;V)$ is compact (compare \eqref{imbedding1}, cf. Aubin's Lemma) then $\mathcal{G}:X\times[0,1]\rightarrow X$ is also compact.\\
Now, for $\zeta=0$ we have $\mathcal{G}(\tilde{\varepsilon},\tilde{m},0)=(0,0)$ for all $(\tilde{\varepsilon},\tilde{m})\in X$. The esimates \eqref{energyestimateepsilon} and \eqref{energyestimatem} imply that the solution of the problem $(\varepsilon,m)=\mathcal{G}(\varepsilon,m,\zeta)$, for some $\zeta\in[0,1]$ is uniformly bounded in $X$. The existence of at least one fixed point to $\mathcal{G}$, i.e. $(\varepsilon_\delta,m_\delta)\in X$ with $\,\mathcal{G}(\varepsilon_\delta,m_\delta,1)=(\varepsilon_\delta,m_\delta)$, follows by the Leray--Schauder Theorem. Consequently, the equation $(\varepsilon_\delta,m_\delta)=\mathcal{G}(\varepsilon_\delta,m_\delta,1)\in V_2^{2,1}(\Omega\times S)$ has a solution to Problem $(P^\delta)$, and so,  Theorem \ref{strongsolutionpdelta} is therefore proven.

\end{proof}

\subsection{Passage to the limit $\delta\rightarrow 0$}

To complete the proof of the main result stated in Theorem \ref{main},  we pass now to the limit for $\delta\rightarrow 0$. In other words, we  study the weak convergence of the solution $(\varepsilon_\delta,m_\delta)$ to problem $(P^\delta)$ to the solution $(\varepsilon,m)$ to problem $(P)$.\\

We write down problem $(P^\delta)$ in the form
\begin{align}
&\frac{\partial\varepsilon_\delta}{\partial t}+\mbox{div}(-k_1\nabla\varepsilon_\delta-k_2\nabla^\delta m_\delta)=\hat{f}_1(m_\delta,\varepsilon_\delta)&\mbox{in}\,\Omega\times S,
\label{deltamolliproblemadipartenzaconepsilonpunto}
\\
&\frac{\partial m_\delta}{\partial t}+\mbox{div}(-k_2\nabla^\delta\varepsilon_\delta-k_3\nabla m_\delta)=\hat{f}_2(m_\delta,\varepsilon_\delta)&\mbox{in}\,\Omega\times S,
\label{2deltamolliproblemadipartenzaconepsilonpunto}
\\
&\varepsilon_\delta(0)=\varepsilon_0&\mbox{in}\,\Omega,
\label{3deltamolliproblemadipartenzaconepsilonpunto}
\\
& m_\delta(0)=m_0&\mbox{in}\,\Omega,
\label{4deltamolliproblemadipartenzaconepsilonpunto}
\\
&\varepsilon_\delta(l_1,t)=\varepsilon_D(t)&\mbox{in}\,S,\label{5deltamolliproblemadipartenzaconepsilonpunto}
\\
&m_\delta(l_1,t)=m_D(t)&\mbox{in}\,S,
\label{6deltamolliproblemadipartenzaconepsilonpunto}
\\
&\frac{\partial \varepsilon_\delta}{\partial x}(l_2,t)=0&\mbox{in}\,S,
\label{7deltamolliproblemadipartenzaconepsilonpunto}
\\
&\frac{\partial m_\delta}{\partial x}(l_2,t)=0&\mbox{in}\,S.
\label{8deltamolliproblemadipartenzaconepsilonpunto}
\end{align}
The next Lemma recapitulates the basic convergences we rely on.
\begin{Lemma}
Under the assumptions $H_1$--$H_4$, the following convergences 
hold up to subsequences, as $\delta\rightarrow 0$:
\begin{enumerate}
\item[(i)] $\varepsilon_\delta\rightharpoonup\varepsilon$ in $L^2(S;V)$,\; $m_\delta\rightharpoonup m$ in $L^2(S;V)$.
\item[(ii)] $\partial_t\varepsilon_\delta\rightharpoonup\partial_t\varepsilon$ in $L^2(S;L^2(\Omega))$,\; $\partial_t m_\delta\rightharpoonup\partial_t m$ in $L^2(S;L^2(\Omega))$.
\item[(iii)] $\varepsilon_\delta\rightarrow\varepsilon$ in $L^2(S;L^2(\Omega))$,\; $ m_\delta\rightarrow m$ in $L^2(S;L^2(\Omega))$.
\item[(iv)] $\hat{f}_1(\varepsilon_\delta,m_\delta)\rightarrow\hat{f}_1(\varepsilon,m)$ a.\,e.\,in\,$\Omega\times S$,\; $\hat{f}_2(\varepsilon_\delta,m_\delta)\rightarrow\hat{f}_2(\varepsilon,m)$ a.\,e.\,in\,$\Omega\times S$.
\item[(v)]  $\nabla^\delta\varepsilon_\delta\rightarrow\nabla\varepsilon$ in $L^2(S;L^2(\Omega))$,\; $\nabla^\delta m_\delta\rightarrow\nabla m$ in $L^2(S;L^2(\Omega))$.
\item[(vi)] $\varepsilon_\delta\overset{*}{\rightharpoonup}\varepsilon$ in $L^\infty(S;L^2(\Omega))$,\; $ m_\delta\overset{*}{\rightharpoonup}m$ in $L^\infty(S;L^2(\Omega))$.
\item[(vii)] If $m_\delta,\varepsilon_\delta\in V_2^{2+\nu,1}(\Omega\times S)$ for $\nu>0$, then $\varepsilon_\delta\overset{*}{\rightharpoonup}\varepsilon$ in $L^\infty(S;L^\infty(\Omega))$,\; $ m_\delta\overset{*}{\rightharpoonup}m$ in $L^\infty(S;L^\infty(\Omega))$.
\end{enumerate}
\label{convergences}
\end{Lemma}
\begin{proof}
$(i)$ simply follows from the estimates \eqref{energyestimateepsilon} and \eqref{energyestimatem}, while $(ii)$ is a direct consequence of the estimates \eqref{stimaepsilot2} and \eqref{stimamt2}. To deal with $(iii)$, we make use of the Aubin's compactness  lemma (see Theorem \ref{aubin}), particularly by choosing 
$$B_0:=V,\;B=L^2(\Omega)\;B_1:=L^2(\Omega).$$
Now, defining $W$ (see \eqref{spazioW}) as 
$$W:=\left\lbrace\varphi\in L^2(S;V),\;\partial_t\varphi\in L^2(S;L^2(\Omega))\right\rbrace,$$
we get
$$W\hookrightarrow\hookrightarrow L^2(S;L^2(\Omega)).$$ 
$(iv)$ simply follows from $(iii)$. $(v)$ is a consequence of Theorem \ref{propmoll}. $(vi)$ follows from the estimates \eqref{energyestimateinftyepsilon1} and \eqref{energyestimateinftym1}. To deal with $(vii)$, we note that $V_2^{2+\nu,1}\hookrightarrow L^\infty(\Omega\times S)$. In fact, we have $H^{2+\nu}(\Omega)\hookrightarrow L^\infty(\Omega),\,\nu>0$, see \cite{fucik} Theorem 5.7.8 p. 287, and hence, we deduce
\begin{eqnarray}
\label{energyestimateinftyepsilon2}
\|\varepsilon_\delta\|_{L^\infty(S;L^\infty(\Omega))}\leq C,\\
\|m_\delta\|_{L^\infty(S;L^\infty(\Omega))}\leq C.
\label{energyestimateinftym2}
\end{eqnarray} 
From \eqref{energyestimateinftyepsilon2} and \eqref{energyestimateinftym2} we obtain $(vii)$.
\end{proof}
It is worth noting that the strong solution  
$\tilde{\varepsilon}\in \zeta\varepsilon_D+V_2^{2,1}(\Omega\times S),\,m\in\zeta m_D+V_2^{2,1}(\Omega\times S)$ ensured by Theorem \ref{strongsolutionpdelta} is a solution of problem $(P^\delta)$ and satisfies the identities
\begin{eqnarray}
&&-\int_\Omega\varepsilon_\delta(x,0)\varphi(x,0)-\int_{\Omega\times S}\varepsilon_\delta\partial_t\varphi dxdt+k_1\int_{\Omega\times S}\nabla\varepsilon_\delta\nabla\varphi dxdt
\nonumber\\
&&+k_2\int_{\Omega\times S}\nabla^\delta m_\delta\nabla\varphi dxdt=\int_{\Omega\times S}\hat{f}_1(\varepsilon_\delta,m_\delta)\varphi dxdt,
\label{deltamustgotozero1}
\end{eqnarray}
\begin{eqnarray}
&&-\int_\Omega m_\delta(x,0)\psi(x,0)-\int_{\Omega\times S}m_\delta\partial_t\psi dxdt+k_3\int_{\Omega\times S}\nabla m_\delta\nabla\psi dxdt
\nonumber\\
&&+k_2\int_{\Omega\times S}\nabla^\delta\varepsilon_\delta\nabla\psi dxdt=\int_{\Omega\times S}\hat{f}_2(\varepsilon_\delta,m_\delta)\psi dxdt
\label{deltamustgotozero2}
\end{eqnarray}
for all test functions $\varphi,\,\psi\in C(\bar{S};C_0(\bar{\Omega}))$ and $\varphi(x,T)=\psi(x,T)=0\;\mbox{for\,all}\,x\in\Omega$.
\\
The convergences $(i)$--$(v)$ established in Lemma \ref{convergences} are sufficient for taking the weak limit $\delta\rightarrow 0$ in \eqref{deltamustgotozero1} and \eqref{deltamustgotozero2}. Thus, we have
\begin{eqnarray}
&&-\int_\Omega\varepsilon(x,0)\varphi(x,0)-\int_{\Omega\times S}\varepsilon\partial_t\varphi dxdt+k_1\int_{\Omega\times S}\nabla\varepsilon\nabla\varphi dxdt
\nonumber\\
&&+k_2\int_{\Omega\times S}\nabla^\delta m\nabla\varphi dxdt=\int_{\Omega\times S}\hat{f}_1(\varepsilon,m)\varphi dxdt,
\label{deltagotozero1}
\end{eqnarray}
\begin{eqnarray}
&&-\int_\Omega m(x,0)\psi(x,0)-\int_{\Omega\times S}m\partial_t\psi dxdt+k_3\int_{\Omega\times S}\nabla m\nabla\psi dxdt
\nonumber\\
&&+k_2\int_{\Omega\times S}\nabla^\delta\varepsilon\nabla\psi dxdt=\int_{\Omega\times S}\hat{f}_2(\varepsilon,m)\psi dxdt
\label{deltagotozero2}
\end{eqnarray} 
for all test functions $\varphi$ and $\psi$ previously chosen.\\
Now, integrating back \eqref{deltagotozero1} and \eqref{deltagotozero2}, we obtain:
\begin{eqnarray}
(\partial_t\varepsilon,\varphi)_{L^2(\Omega)}+k_1(\nabla\varepsilon,\nabla\varphi)_{L^2(\Omega)}=-k_2(\nabla m,\nabla\varphi)_{L^2(\Omega)}+(\hat{f}_1,\varphi)_{L^2(\Omega)},
\end{eqnarray}
\begin{eqnarray}
(\partial_t m,\psi)_{L^2(\Omega)}+k_3(\nabla m),\nabla\psi)_{L^2(\Omega)}=-k_2(\nabla\varepsilon,\nabla\psi)_{L^2(\Omega)}+(\hat{f}_2,\psi)_{L^2(\Omega)},
\end{eqnarray}
which is precisely our concept of weak solution to Problem $(P)$, see Defintion \ref{weaksol}. This completes the proof of the main result Theorem \ref{main}.

\section{Proof of Theorem \ref{uniqueness}}
Let us suppose that problem $(P)$
has two different solutions $(\varepsilon_1,m_1)$ and $(\varepsilon_2,m_2)$ endowed with the same initial conditions. Define $w_\varepsilon:=\varepsilon_1-\varepsilon_2$ and $w_m:=m_1-m_2$, then $(w_\varepsilon,w_m)$ satisfy, for all $\varphi,\psi\in V$, the following identities:
\begin{eqnarray}
(\partial_t w_\varepsilon,\varphi)_{L^2(\Omega)}+(\mbox{div}(-k_1\nabla w_\varepsilon),\varphi)_{L^2(\Omega)}+(\mbox{div}(k_2\nabla w_m),\varphi)_{L^2(\Omega)}
\nonumber\\=(\hat{f}_1(\varepsilon_1,m_1)-\hat{f}_1(\varepsilon_2,m_2),\varphi)_{L^2(\Omega)},
\label{uniq1phi}
\end{eqnarray}
and
\begin{eqnarray}
(\partial_t w_m,\psi)_{L^2(\Omega)}+(\mbox{div}(-k_3\nabla w_m),\psi)_{L^2(\Omega)}+(\mbox{div}(k_2\nabla w_\varepsilon),\psi)_{L^2(\Omega)}
\nonumber\\=(\hat{f}_2(\varepsilon_1,m_1)-\hat{f}_2(\varepsilon_2,m_2),\psi)_{L^2(\Omega)}.
\label{uniq1psi}
\end{eqnarray}
We choose $\varphi:=w_\varepsilon\in V,\,\psi:=w_m\in V$, then we obtain
\begin{eqnarray}
\frac{1}{2}\frac{d}{dt}\|w_\varepsilon\|^2_{L^2(\Omega)} +k_1\|\nabla w_\varepsilon\|^2_{L^2(\Omega)}=-k_2\int_\Omega\nabla w_\varepsilon\nabla w_m dx+\int_\Omega(\hat{f}_1(\varepsilon_1,m_1)-\hat{f}_1(\varepsilon_2,m_2))w_\varepsilon dx,
\nonumber\\
\end{eqnarray}
and hence,
\begin{eqnarray}
\frac{1}{2}\frac{d}{dt}\|w_m\|^2_{L^2(\Omega)} +k_1\|\nabla w_m\|^2_{L^2(\Omega)}=-k_2\int_\Omega\nabla w_\varepsilon\nabla w_m dx
\nonumber\\+\int_\Omega(\hat{f}_2(\varepsilon_1,m_1)-\hat{f}_2(\varepsilon_2,m_2))w_m dx.
\end{eqnarray}
Applying the geometric mean--aritmetic mean inequality to the first term of the right hand side of both the above equations, we obtain 
\begin{eqnarray}
\frac{1}{2}\frac{d}{dt}\|w_\varepsilon\|^2_{L^2(\Omega)} +k_1\|\nabla w_\varepsilon\|^2_{L^2(\Omega)}\leq \frac{k_2}{2}\|\nabla w_\varepsilon\|^2_{L^2(\Omega)} +\frac{k_2}{2}\|\nabla w_m\|^2_{L^2(\Omega)} 
\nonumber\\+\int_\Omega(\hat{f}_1(\varepsilon_1,m_1)-\hat{f}_1(\varepsilon_2,m_2))w_\varepsilon dx,
\label{1dasumuniq}
\end{eqnarray}
\begin{eqnarray}
\frac{1}{2}\frac{d}{dt}\|w_m\|^2_{L^2(\Omega)}+k_3\|\nabla w_m\|^2_{L^2(\Omega)}\leq \frac{k_2}{2}\|\nabla w_\varepsilon\|^2_{L^2(\Omega)} +\frac{k_2}{2}\|\nabla w_m\|^2_{L^2(\Omega)}
\nonumber\\+ \int_\Omega(\hat{f}_2(\varepsilon_1,m_1)-\hat{f}_2(\varepsilon_2,m_2))w_m dx.
\label{2dasumuniq}
\end{eqnarray}
We observe that, due to the particular structure of $\hat{f}_1$ and $\hat{f}_2$ from \eqref{f1cappello}--\eqref{f1f2}, there exist the constants $A_1,\,B_1,\,A_2,\,B_2>0$ such that
\begin{eqnarray}
\left|\hat{f}_1(\varepsilon_1,m_1)-\hat{f}_1(\varepsilon_2,m_2)\right|
\leq A_1|w_\varepsilon|+B_1|w_m|,
\label{1uniqf1}
\end{eqnarray}
\begin{eqnarray}
\left|\hat{f}_2(\varepsilon_1,m_1)-\hat{f}_2(\varepsilon_2,m_2)\right|
\leq A_2|w_\varepsilon|+B_2|w_m|.
\label{2uniqf2}
\end{eqnarray}
Now, we sum up \eqref{1dasumuniq} and \eqref{2dasumuniq} and we use \eqref{1uniqf1}-\eqref{2uniqf2} to get:
\begin{eqnarray}
&&\frac{1}{2}\frac{d}{dt}\|w_\varepsilon\|^2_{L^2(\Omega)}+\frac{1}{2}\frac{d}{dt}\|w_m\|^2_{L^2(\Omega)}+(k_1-k_2)\|\nabla w_\varepsilon\|^2_{L^2(\Omega)}+(k_3-k_2)\|\nabla w_m\|^2_{L^2(\Omega)}
\nonumber\\
&&\leq A\|w_\varepsilon\|^2_{L^2(\Omega)}+B\|w_m\|^2_{L^2(\Omega)}.
\end{eqnarray}
Using assumption $H_5$, we have 
\begin{eqnarray}
&&\frac{1}{2}\frac{d}{dt}\|w_\varepsilon\|^2_{L^2(\Omega)}+\frac{1}{2}\frac{d}{dt}\|w_m\|^2_{L^2(\Omega)}\leq K(\|w_\varepsilon\|^2_{L^2(\Omega)}+\|w_m\|^2_{L^2(\Omega)}),
\end{eqnarray}
where $K=\mbox{max}\left\lbrace A,B\right\rbrace$. Using now Gronwall's inequality, we get for all $t\in S$
\begin{eqnarray}
0\leq \|w_\varepsilon(t)\|^2_{L^2(\Omega)}+\|w_m(t)\|^2_{L^2(\Omega)}\leq e^{2Kt}\left(\|w_\varepsilon(0)\|^2_{L^2(\Omega)}+\|w_m(0)\|^2_{L^2(\Omega)}\right),
\end{eqnarray}
where $w_\varepsilon(0)=w_\varepsilon(0)=0$. Thus
$\varepsilon_1=\varepsilon_2,\,m_1=m_2$ a.e. in $\Omega$ for all $t\in S$.



\section{Negativity and boundedness of the strain and density:
proof of Theorem \ref{boundandnegative}}
\label{sec:negativity}
Due to the cross--diffusion--like structure the Problem $(P)$ does not admit a weak maximum principle. Also, we were not able to find suitable test functions to obtain the boundedness or the negativity\footnote{It is worth noting that testing with $\varphi=\varepsilon^+\in V$ and $\psi=m^+\in V$ (or $\varphi=(\varepsilon-M_1)^+$ and $\psi=(m-M_2)^+$ to search for boundedness) does not work. This is mainly because we can not control the sign of terms like $k_2\nabla\varepsilon^-\nabla m^+$ or $k_2\nabla^\delta\varepsilon^-\nabla m^+$.
} by means of an energy--like argument. In what follows, we rely on a regularity argument to ensure the boundedness property and use finite difference scheme to detect the negativity of the solution. 
 
The boundedness of the solution $(\varepsilon,m)$ to Problem $(P)$ is obtained based on the additional regularity indicated in Lemma \ref{convergences} $(vi)$ (together with $H_6$).

To prove the negativity of $(\varepsilon,m)$ we proceed as follows: 
Due to the regularity of the solution, in one--space dimension the functions $\varepsilon$ and $m$ are continuous. We then use the approach (and the corresponding notation)  from \cite{saito} to construct using finite difference approximations a non--positive subsequence $(\varepsilon^n,m^n)$ convergent to $(\varepsilon,m)$, situation which is valid provided the negativity of the functions $f_1,f_2$ is guaranteed.
\\
We rewrite Problem $(P)$ in the following way:

\begin{align}
&\partial_t m=(k_3\partial_x m+k_2\partial_x\varepsilon)_x+\hat{f}_2(\varepsilon,m) &\mbox{in}\,\Omega\times S,
\label{perdiffscheme}\\
&\partial_t\varepsilon=k_1( \partial_x\varepsilon+k_2\partial_x m)_x+\hat{f}_1(\varepsilon,m)&\mbox{in}\,\Omega\times S,
\label{1perdiffscheme}\\
&\varepsilon(x,0)=\varepsilon_0(x)&\mbox{in}\,\Omega,\\
&m(x,0)=m_0(x)&\mbox{in}\,\Omega,\\
&\varepsilon(l_1,t)=\varepsilon_D&\mbox{in}\,S,\\
&m(l_1,t)=m_D&\mbox{in}\,S,\\
&\partial_x \varepsilon(l_2,t)=0&\mbox{in}\,S,
\\&\partial_x m(l_2,t)=0&\mbox{in}\,S,
\end{align}
where we have $\hat{f}_1,\,\hat{f}_2\leq 0$.\\
Take now a positive integer $N$ and let $h:=1/N$. We introduce two types of grid points as
\begin{eqnarray}
x_i:=(i-\frac{1}{2})h,\,i\in\left\lbrace 0,\dots,N+1\right\rbrace\; \mbox{and}\quad \hat{x}_i=ih,\,i\in\left\lbrace -1,\dots,N+1\right\rbrace\,
\end{eqnarray}
and set sub--intervals
\begin{eqnarray}
I_i:=(\hat{x}_{i-1},\hat{x}_i),\,i\in\left\lbrace 1,\dots,N\right\rbrace\;\mbox{and}
\quad 
\hat{I}_i:=(x_i,x_{i+1}),\,i\in\left\lbrace 0,\dots,N+\right\rbrace.
\end{eqnarray}
Furthermore, we set
\begin{eqnarray}
X_h:=\left\lbrace\sum_{i=1}^Nc_i\chi_i,\;\left\lbrace c_i\right\rbrace_{i=1}^N\subset\mathbb{R}\right\rbrace\;\mbox{and}\quad \hat{X}_h:=\left\lbrace\sum_{i=0}^N\hat{c}_i\hat{\chi}_i,\;\left\lbrace \hat{c}_i\right\rbrace_{i=0}^N\subset\mathbb{R}\right\rbrace,
\end{eqnarray}
where $\chi_i$ and $\hat{\chi}_i$ denote the characteristic functions of $I_i$ and $\hat{I}_i\cap[0,1]$ respectively.\\
Let $\left\lbrace\tau_n\right\rbrace_{i=1}^\sigma$ be a set of positive numbers and suppose that the $n$--th time step $t_n$ is determined by
\begin{eqnarray}
t_0=0,\qquad t_n=t_{n-1}+\tau_n=\sum_{k=1}^n\tau_k,\,i\in\left\lbrace 1,\dots,\sigma\right\rbrace\,\quad t_m\leq T.
\end{eqnarray}
With such notation available, we find the unknown functions in the following form:
\begin{eqnarray}
m^n_h:=\sum^N_{i=1}m_i^n\chi_i\approx m(x,t_n)\\
b^n_h:=\sum^N_{i=1}b_i^n\chi_i\approx \varepsilon_x(x,t_n)\\
F^n_h:=\sum^N_{i=0}F_i^n\hat{\chi}_i\approx F^n\equiv(k_3\partial_x m+k_2\partial_x\varepsilon)\\
\varepsilon^n_h:=\sum^N_{i=0}\varepsilon_i^n\hat{\chi}_i\approx \varepsilon(x,t_n)
\end{eqnarray}
for $n=\left\lbrace 0,\dots,\sigma\right\rbrace$.\\
Firstly, we introduce a difference scheme for equation \eqref{1perdiffscheme}, viz.
\begin{eqnarray}
\label{due}
\frac{\varepsilon_i^n-\varepsilon_i^{n-1}}{\tau_n}=k_1\frac{\varepsilon_{i-1}^{n-1}-2\varepsilon_i^{n-1}+\varepsilon_{i+1}^{n-1}}{h^2}+k_2\frac{m_{i-2}^{n-1}-2m_{i-1}^{n-1}+m_i^{n-1}}{h^2}+\hat{f}_1(\varepsilon
_i^{n-1},m_i^{n-1})\\\nonumber
i=3,\dots, N\quad\mbox{and}\; \varepsilon_{-1}=\varepsilon_1,\,\varepsilon_{N+1}=\varepsilon_{N-1},
\end{eqnarray}
with $\varepsilon_0=\varepsilon_D<0,\,\varepsilon_1<0$ prescribed.\\
Now we describe a difference scheme for \eqref{perdiffscheme}. We suppose that $m_h^{n-1}$ is known from the time step $t_{n-1}$. Then $\varepsilon_h^n$ can be calculated by \eqref{due}. \\
We compute $b_h^{n-1}$ by
\begin{eqnarray}
b_i^{n}=\frac{\varepsilon_i^n-\varepsilon^n_{i-1}}{h},\qquad i\in\left\lbrace 0,\dots,N+1\right\rbrace.
\end{eqnarray} 
Then we approximate the flux $F_i^n$ is by
\begin{eqnarray}
F_i^n=k_3\frac{m_{i+1}^n-m_i^n}{h}+k_2b_i^n,\qquad i=1,\dots,N-1\\
F_0^n=0,\qquad  F_0^N=0.
\label{bcscheme}
\end{eqnarray}
Our proposed scheme reads then as follows:
\begin{eqnarray}
\frac{m_i^n-m_i^{n-1}}{\tau_n}=\theta\frac{F_i^n-F_{i-1}^n}{h}+(1-\theta)\frac{F_i^{n-1}-F_{i-1}^{n-1}}{h}+\hat{f}_2(\varepsilon^{n-1}_i,m^{n-1}_i),
\\\nonumber i=1,\dots,N,
\label{scheme}
\end{eqnarray}
with the boundary condition \eqref{bcscheme}, $\theta\in[0,1]$ and $m_0=m_D$ prescribed.\\
Now we introduce the matrix representation of \eqref{scheme} and \eqref{bcscheme}. To this aim, setting $\lambda_n=\tau_n/h^2$, we define the $N\times N$ matrix $\textbf{H}=[H_{k,l}]$ by
\begin{eqnarray}
H_{k,l}:=k_3\cdot
\begin{cases}
-1,\qquad k=l=1,\\
1\qquad k=1,\,l=2,\\
1\qquad 2\leq k\leq N-1,\,l=k-1,\\
-2\qquad 2\leq k\leq N-1,\,l=k\\
1\qquad 2\leq k\leq N-1,\,l=k+1,\\
1\qquad k=N,\,l=N-1,\\
-1\qquad k=N,\,l=N,\\
0\qquad \mbox{otherwise},
\end{cases}
\end{eqnarray}
and the $N\times N$ matrix $\textbf{B}_n=[B_{k,l}]$ by
\begin{eqnarray}
B_{k,l}:=hk_2\cdot
\begin{cases}
-b_1,\qquad k=l=1,\\
b_2\qquad k=1,\,l=2,\\
0\qquad 2\leq k\leq N-1,\,l=k-1,\\
-b_k\qquad 2\leq k\leq N-1,\,l=k\\
b_{k+1}\qquad 2\leq k\leq N-1,\,l=k+1,\\
b_N\qquad k=N,\,l=N-1,\\
B_N\qquad k=N,\,l=N,\\
0\qquad \mbox{otherwise}.
\end{cases}
\end{eqnarray}
Then \eqref{scheme} and \eqref{bcscheme} is reduced to
\begin{eqnarray}
(\textbf{I}-\lambda_n\theta\textbf{H})\textbf{m}^n-
\textbf{B}_n\textbf{1}=(\textbf{I}+\lambda_n(1-\theta)\textbf{H})\textbf{m}^{n-1}+
\textbf{B}_{n-1}\textbf{1}+\hat{f}_2\textbf{I}\textbf{1}, \qquad n=\left\lbrace 1,2,\dots,\sigma\right\rbrace,
\label{scrittoamatrix}
\end{eqnarray}
where $\textbf{m}^n=^T(m_1^n,\dots,m_N^n)$, $\textbf{I}$ denotes the identity matrix and $\textbf{1}=^T(1,\dots,1)$.
\\
We state now some assumptions in order to prove the follwoing Theorem \ref{thmdiff} which constructs teh sequence $(\varepsilon^n,m^n)$ needed in the proof of Theorem \ref{boundandnegative}.
\begin{itemize}
\item[$A_1$:] $2\tau_n\theta\leq h$, 
\item[$A_2$:] $2\tau_n(1-\theta)\leq h^2$,
\item[$A_3$:] $b_i^0-b_{i-1}^0+b_i^1-b_{i-1}^1\leq 0$.
\end{itemize}
$A_1$--$A_3$ are all technical assumptions.
\begin{thm}
\label{thmdiff}
Let $n\in\left\lbrace 1,\dots,\sigma\right\rbrace$ and $m_h^{n-1}=\sum_{i=1}^Nm_i^{n-1}\chi_i$, $\varepsilon_h^{n-1}=\sum_{i=0}^N\varepsilon_i^{n-1}\hat{\chi}_i$  be given. Assume that $m_h^{n-1}\leq 0$, $\varepsilon_h^{n-1}\leq 0$ and $m_h^{n-1}$, $\varepsilon_h^{n-1}$ are not identically constant. Assume $A_1$--$A_3$ and define $$\rho(\theta,h):=\min\left\lbrace\frac{h^2}{2(1-\theta)},\frac{h}{2\theta}\right\rbrace,\qquad \iota(h,k_2):=\frac{h^2}{2k_2}.$$\\
Then if 
\begin{eqnarray}
\label{condlambda}
\tau_n\leq\min\left\lbrace\rho(\theta,h),\iota(h,k_2)\right\rbrace,
\end{eqnarray}
the scheme
\eqref{scheme} and \eqref{bcscheme} together with \eqref{due} admits a unique solution $$(\varepsilon_h^n,m_h^n)=(\sum^N_{i=0}\varepsilon_i^n\hat{\chi}_i,\sum_{i=1}^Nm_i^n\chi_i)$$
which satisfies $\varepsilon_i^n<0,\mbox{and}\,m_j^n<0$ for all $i\in\left\lbrace 0,\dots,N\right\rbrace,\,j\in\left\lbrace1,\dots,N\right\rbrace$.
\end{thm}

\begin{proof}
By assumption $A_1$, defining $\textbf{A}:=\textbf{I}-\lambda_n\theta\textbf{H}=[A_{k,l}]$ we can observe that the matrix $A$ is irreducible and diagonally dominant. In fact the irreducibility is a consequence of
\begin{eqnarray}
A_{k,k}>0,\qquad 1\leq k\leq N,
\end{eqnarray}
\begin{eqnarray}
A_{k,k-1}<0,\qquad 2\leq k\leq N,\qquad A_{k,k+1}<0,\qquad 1\leq k\leq N-1.
\end{eqnarray}
On the other hand, by assumption $A_1$, we have
$$\sum_{l=1}^NA_{k,l}\geq 1$$
for $2\leq k\leq N-1$. In a similar way, we have
$$
\sum_{l=1}^NA_{1,l}\geq\frac{1}{2},\qquad\sum_{l=1}^NA_{N,l}\geq\frac{1}{2}.
$$
Thus ${A}$ is diagonally dominant. From asssumption $A_2$, by a direct calculation it is possible to verify that every entry of $(\textbf{I}-\lambda_n(1-\theta)\textbf{H})$ is 
nonnegative.\\
Now we write \eqref{scrittoamatrix} in the following way:
\begin{eqnarray}
(\textbf{I}-\lambda_n\theta\textbf{H})\textbf{m}^n=(\textbf{I}+\lambda_n(1-\theta)\textbf{H})\textbf{m}^{n-1}+
\textbf{B}_{n-1}\textbf{1}+\hat{f}_2\textbf{I}\textbf{1}+
\textbf{B}_n\textbf{1},
\label{scrittoamatrix2}
\end{eqnarray}
and we exploit $A_3$ and \eqref{condlambda} to verify the negativity of the right hand side of \eqref{scrittoamatrix2} and the negativity of $\varepsilon_h^n$. We procede by induction:
\begin{eqnarray}
b_i^0-b_{i-1}^0+b_i^1-b_{i-1}^1\leq 0
\end{eqnarray} is assumption $A_3$. We suppose that 
\begin{eqnarray}
b_i^{n-1}-b_{i-1}^{n-1}+b_i^n-b_{i-1}^n\leq 0
\end{eqnarray}
holds and we prove the same inequality holds also for $n$. By \eqref{scheme}, we have
\begin{eqnarray}
&&b_i^n-b_{i-1}^n+b_i^{n+1}-b_{i-1}^{n+1}=
\nonumber\\&&b_i^n-b_{i-1}^n+\frac{1}{k_2}\left(m_i^{n+1}-m_i^{n}-\tau_n\theta\frac{1}{h^2}\left[k_3(m_{i+1}^n-2m_i^n+m_{i-1}^n)\right]\right)\\
\nonumber\\
&&-\frac{1}{k_2}\left(\tau_n(1-\theta)\frac{1}{h^2}\left[k_3(m_{i+1}^{n-1}-2m_i^{n-1}+m_{i-1}^{n-1})\right]\right)-(b_i^n-b_{i-1}^n)\leq 0.
\end{eqnarray}
Morover we know that $f_2$ is negative. Thus, we have 
\begin{eqnarray}
m^n<0\quad\mbox{for\,all}\,1\leq i\leq N.
\label{negativitàm}
\end{eqnarray}
Take now $n=1$ in \eqref{due}. Using both \eqref{negativitàm} and the negativity of $\hat{f}_1$, we easily obtain $\varepsilon^0\leq 0$. Consider again \eqref{due} and suppose that $\varepsilon_{n-1}\leq 0$, again by a direct calculation it holds $\varepsilon_n\leq 0$.
\end{proof}
Theorem \ref{boundandnegative} is now proven.

\section{Numerical study of steady states of strains and fluid densities for the 
consolidation problem}
\label{s:num}
As we already mentioned above, a very interesting application of 
the theory developed (recalled) in Section~\ref{s:cons} is the study
of profile formation in porous media in a phase transition regime. We will consider a system exhibiting
two phases differing in the strain $\varepsilon$ and in the fluid content $m$. In this situation on a finite one--dimensional bar the system can show profiles, in $\varepsilon$ and $m$, connecting one phase to the other. 

We consider the following expression for the total potential 
energy density in the perspective of describing the transition between
a fluid--poor and a fluid--rich phase 
\begin{equation}
\label{sec015}
\Psi(m,\varepsilon)
 :=
 \frac{\alpha}{12}m^2(3m^2\!-8b\varepsilon m+6b^2\varepsilon^2)
 +
 \Psi_\textrm{B}(m,\varepsilon),
\end{equation}
where 
\begin{equation}
\label{sec020}
\Psi_\textrm{B}(m,\varepsilon):=
 p\varepsilon+\frac{1}{2}\varepsilon^2+\frac{1}{2}a(m-b\varepsilon)^2
\end{equation}
is the Biot potential energy density~\cite{biot01},
$a>0$ is the ratio between the fluid and the solid rigidity, 
$b>0$ is a coupling between the fluid and the solid component, 
$p>0$ is the external pressure,
and
$\alpha>0$ is a material parameter responsible for the showing 
up of an additional equilibrium.

In the papers~\cite{CIS2011,CIS2010} we have studied 
the stationary version of the problem (\ref{gd03}) 
corresponding to the potential energies (\ref{sec010}) and (\ref{sec015})
to describe the possible occurance of an interface between two phases 
differing in fluid content. In fact this model is built in such a way to 
describe the existence of two states of equilibrium: 
the fluid--poor phase $(\varepsilon_s,m_s)$ and the fluid--rich 
phase $(\varepsilon_f,m_f)$ corresponding to the 
two minima 
of the double--well potential energy $\Psi$ in \eqref{sec015}.
Note that for $a=0.5,\,b=1,\alpha=100$ the pressure ensuring the existence 
of two phases is $p=0.24221$, while for the two phases we find 
$\varepsilon_s=-0.1436,\,m_s=-0.1436,\,\varepsilon_f=-0.1598,\,m_f=-0.0427$.

The dissipative dynamics \eqref{gd03} of the porous 
medium model, 
starting from any initial state $(\varepsilon_0,m_0)$, 
leads the system to a stationary state, which is the solution of the 
following problem:
\begin{align}
&(-k_1\varepsilon'-k_2 m')'=f_1(m,\varepsilon)&\mbox{in}\,\Omega,
\label{problemastazdipartenzaconepsilonpunto}
\\
&(-k_2\varepsilon'-k_3m')'=f_2(m,\varepsilon)&\mbox{in}\,\Omega,
\label{2problemastazdipartenzaconepsilonpunto}
\\
&\varepsilon(l_1)=\varepsilon_D,
\label{5problemastazdipartenzaconepsilonpunto}
\\
& m(l_1)=m_D,&
\label{6problemastazdipartenzaconepsilonpunto}
\\
&\frac{\partial \varepsilon}{\partial x}(l_2)=0,
\label{7problemastazdipartenzaconepsilonpunto}
\\
&\frac{\partial m}{\partial x}(l_2)=0.
\label{8problemastazdipartenzaconepsilonpunto}
\end{align}
with
\begin{eqnarray}
\begin{cases}
{\displaystyle
f_1(m,\varepsilon)
:=
-\frac{\partial\Psi}{\partial\varepsilon}(m,\varepsilon)
=
\frac{2}{3}b\alpha m^3-\alpha b^2m^2\varepsilon
                -p-\varepsilon+abm-ab^2\varepsilon
\vphantom{\bigg\{_\}}}\\
{\displaystyle
f_2(m,\varepsilon):=
-\frac{\partial\Psi}{\partial m}(m,\varepsilon)
=
-\alpha m^3+2\alpha b\varepsilon m^2-b^2\alpha \varepsilon^2 m-am
+ab\varepsilon}
\end{cases}
\label{fcap}
\end{eqnarray}
where we recall \eqref{sec015}.
Note that this system of stationary equations has the same form 
of the stationary problem corresponding to our general 
problem $(P)$ introduced in Section~\ref{Strong}.

In our case of Dirichlet--Neumann boundary conditions, the stationary Problem \eqref{problemastazdipartenzaconepsilonpunto}--\eqref{8problemastazdipartenzaconepsilonpunto} has not a unique solution. From the physical point of view, this property means that it is possible to observe different strain and fluid content stationary profiles with the same Dirichlet condition $(\varepsilon_D,m_D)$ at one end. Below we discuss some graphs representing the solution of equations \eqref{problemastazdipartenzaconepsilonpunto}--\eqref{8problemastazdipartenzaconepsilonpunto} in the interval $\Omega=[0,1]$.

\begin{figure}
\begin{picture}(200,400)(15,0)
\put(0,300)
{
\resizebox{6.5cm}{!}{\rotatebox{0}{\includegraphics{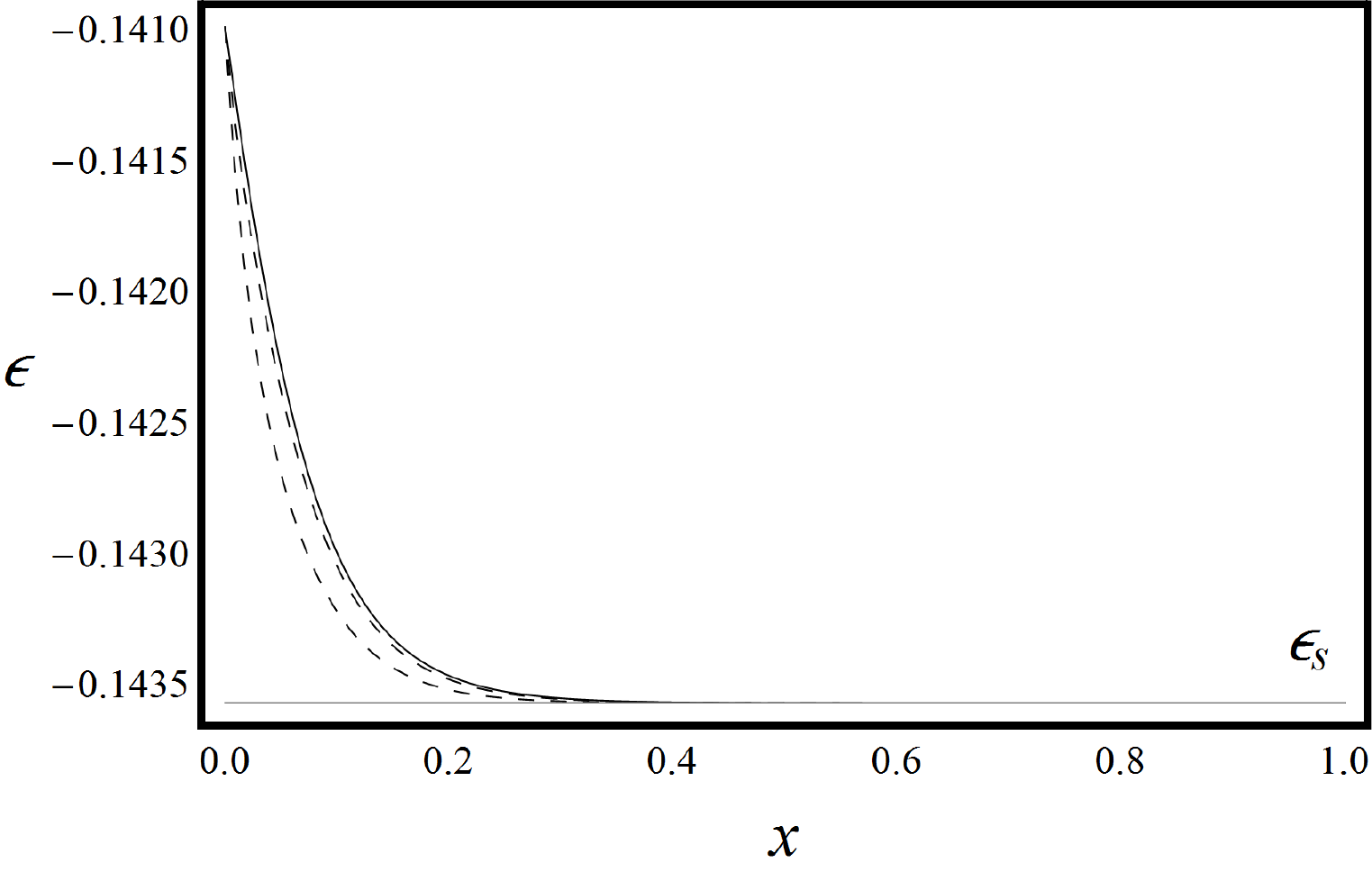}}}
}
\put(215,300)
{
\resizebox{6.5cm}{!}{\rotatebox{0}{\includegraphics{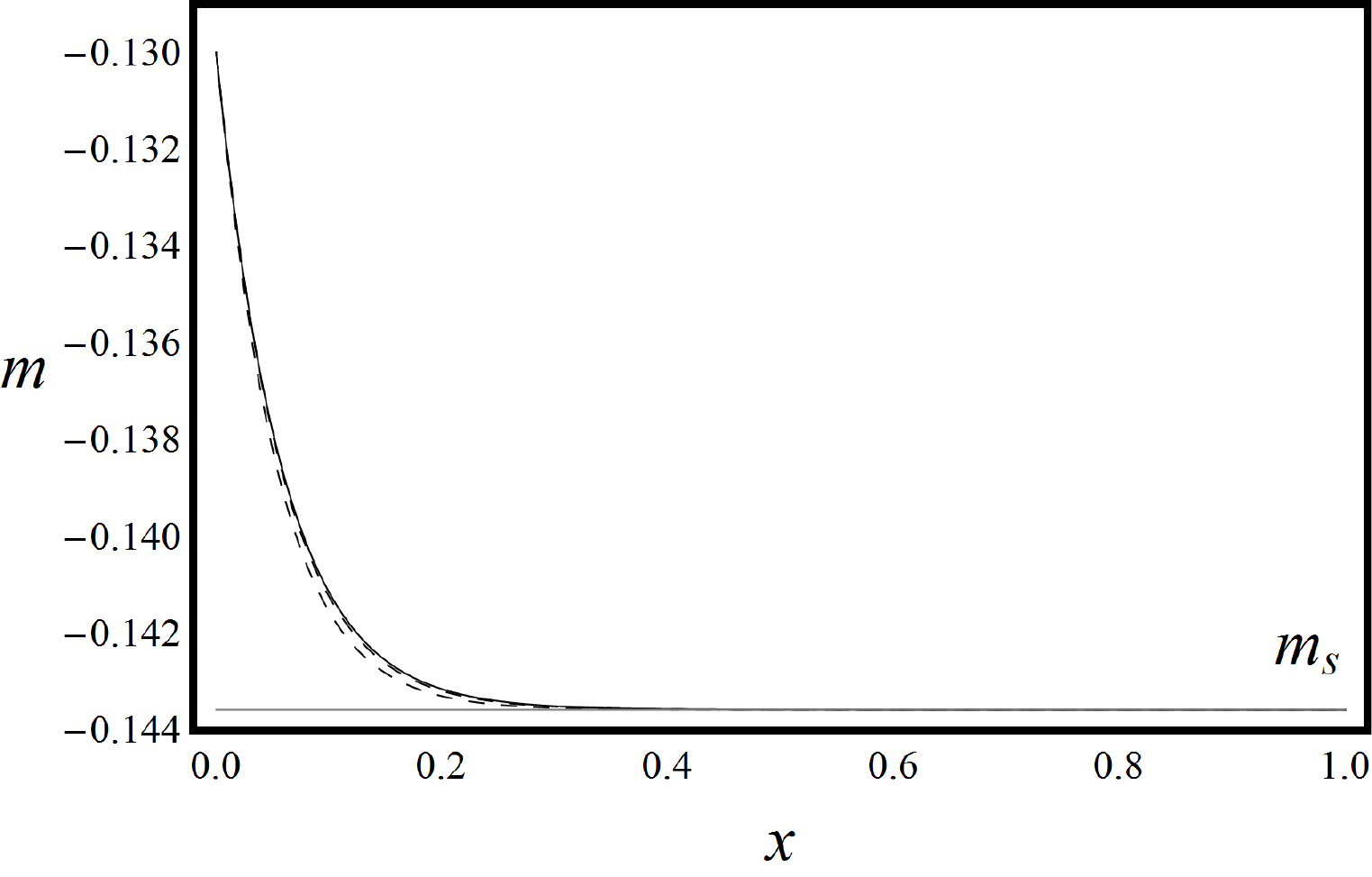}}}
}
\put(0,150)
{
\resizebox{6.5cm}{!}{\rotatebox{0}{\includegraphics{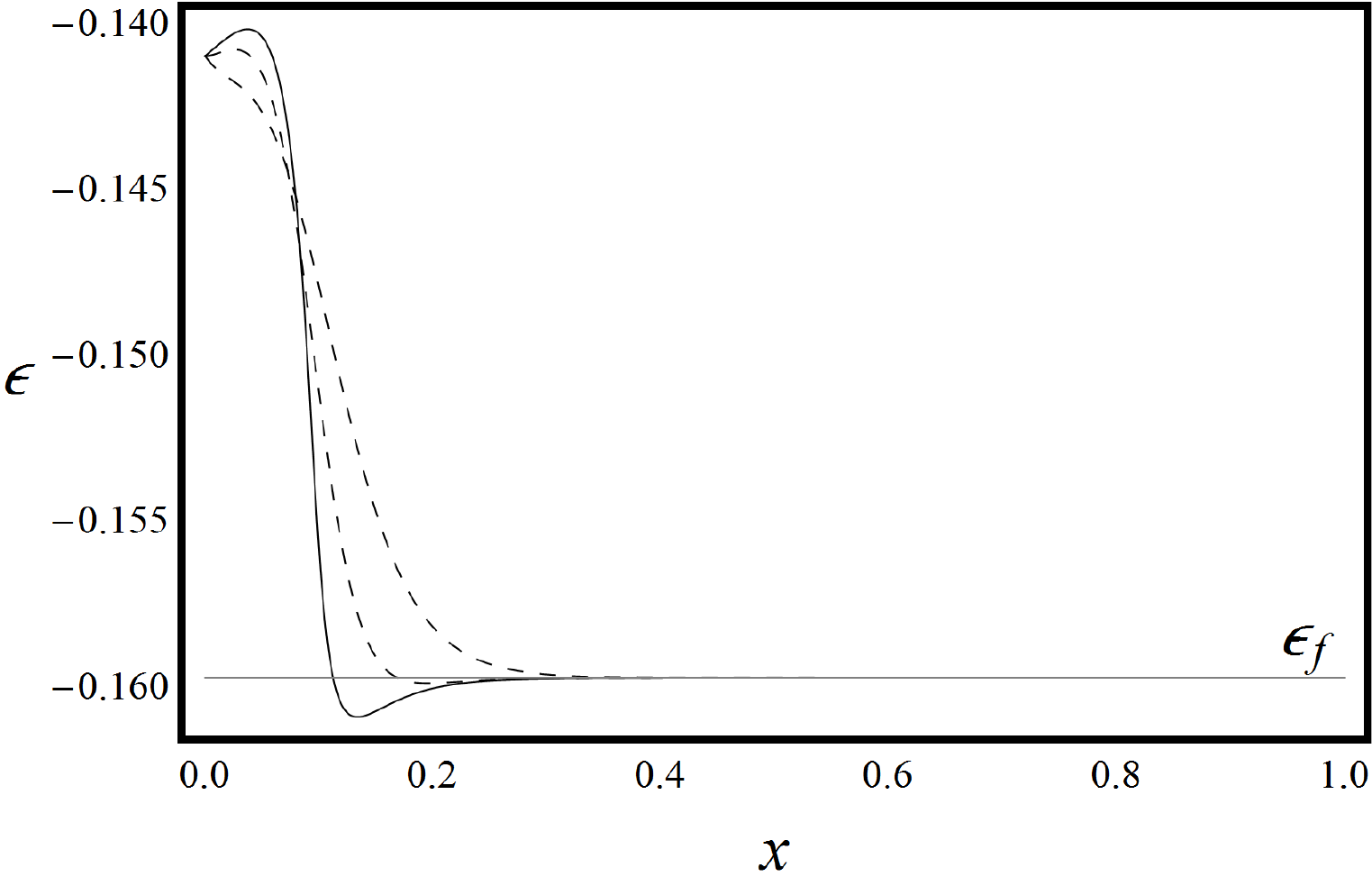}}}
}
\put(215,150)
{
\resizebox{6.5cm}{!}{\rotatebox{0}{\includegraphics{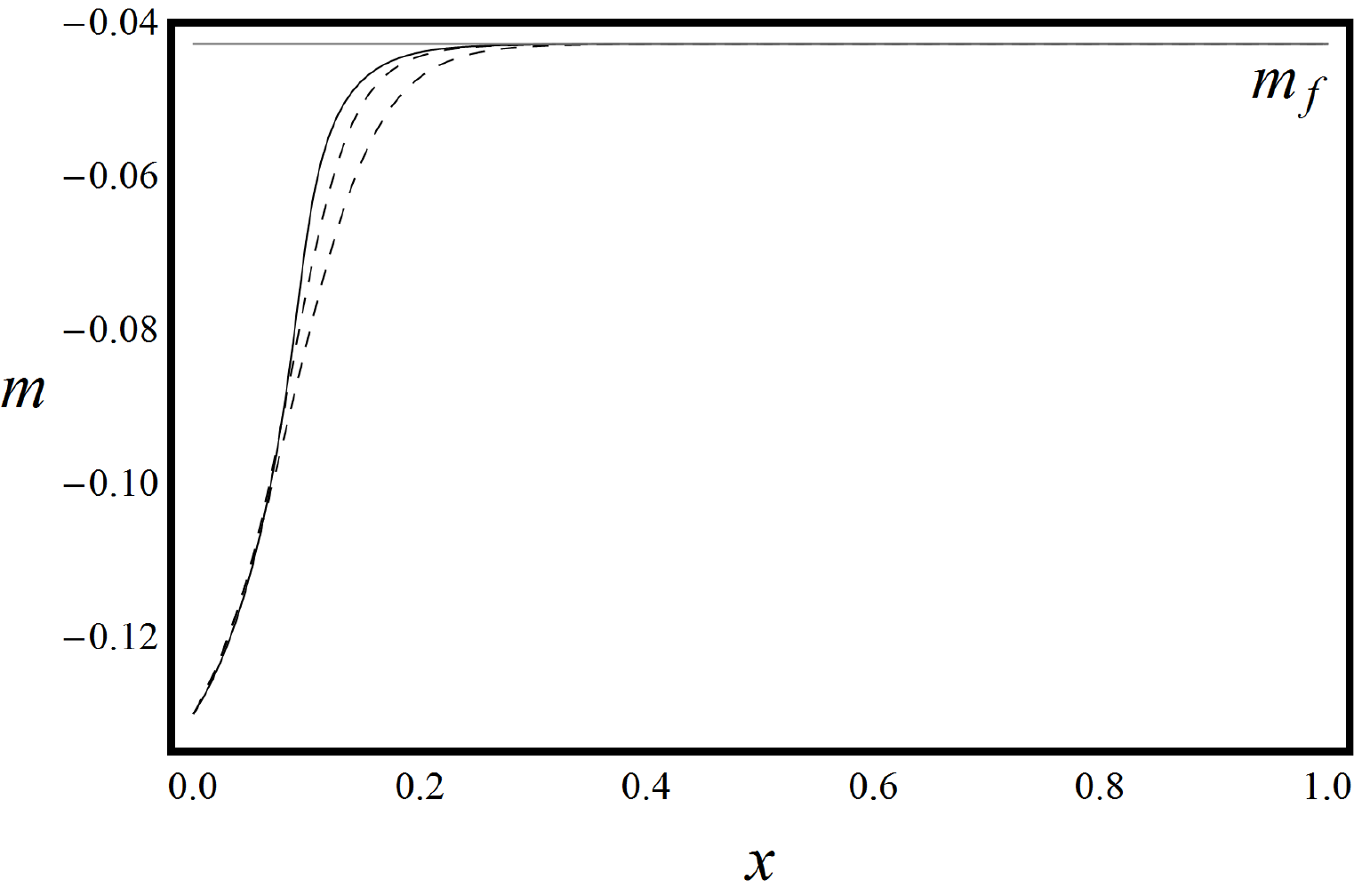}}}
}
\put(0,0)
{
\resizebox{6.5cm}{!}{\rotatebox{0}{\includegraphics{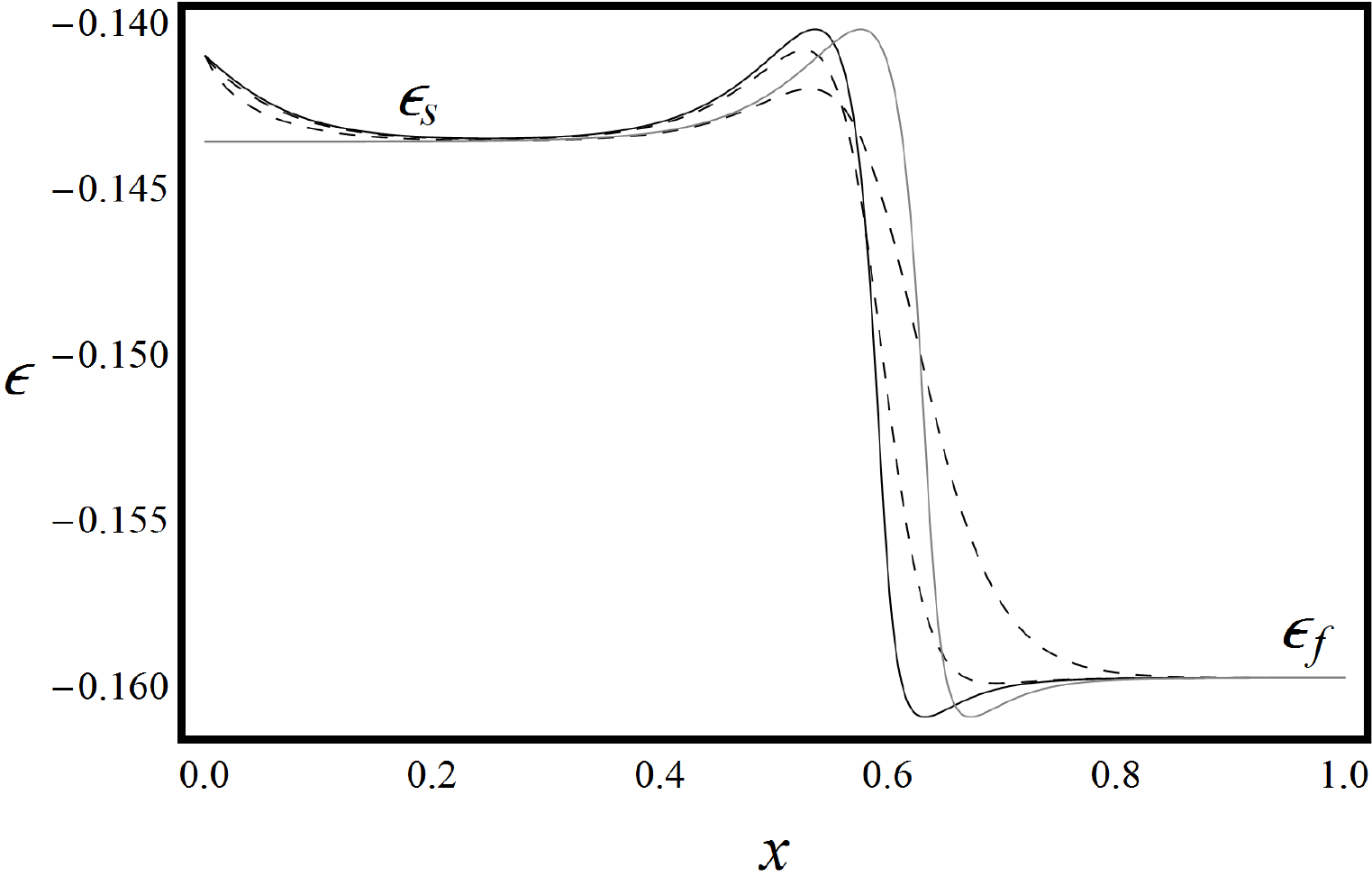}}}
}
\put(215,0)
{
\resizebox{6.5cm}{!}{\rotatebox{0}{\includegraphics{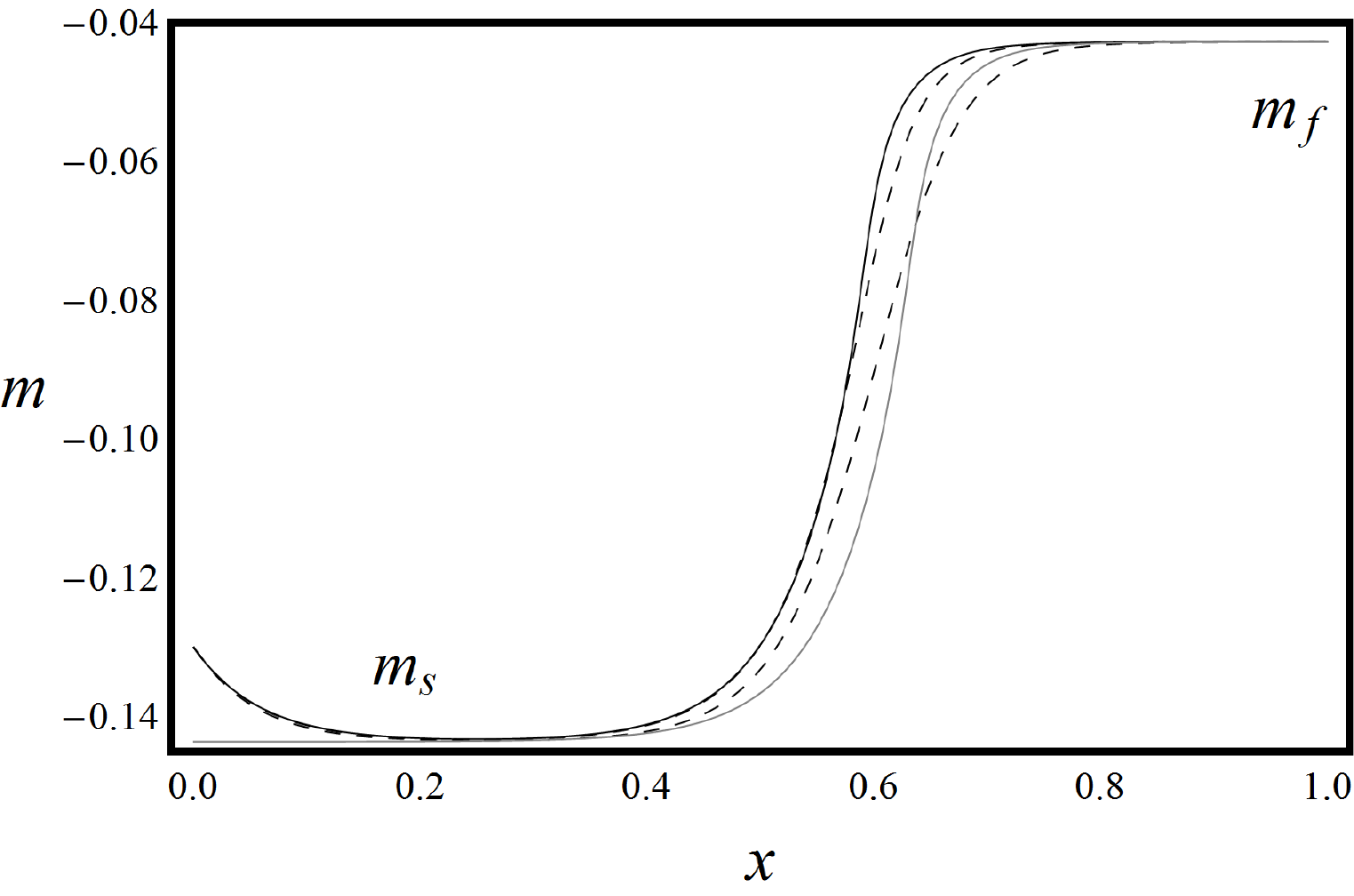}}}
}
\end{picture}
\vskip 1. cm
\centering
\caption{Solutions $\varepsilon(x)$ (left) and $m(x)$ (right) of the stationary problem \eqref{problemastazdipartenzaconepsilonpunto}--\eqref{8problemastazdipartenzaconepsilonpunto} with the boundary conditions $\varepsilon(0)=\bar{\varepsilon}=-0.141,\, m(0)=\bar{m}=-0.13,\,\partial_x \varepsilon(x=1)=0,\,and\,\partial_x m(x=1)=0$ on the finite interval $[0,1]$, for $p=0.24,\,a=0.5,\,b=1,\alpha=100,k_1=10^{-3}=10^{-3},\,k_2=10^{-3}$ (solid line), and $k_1=10^{-3}=10^{-3},\,k_2=0.2\times 10^{-3},\,0.8\times 10^{-3}$ (dotted lines), starting by the following initial guesses (gray lines): costant fluid--poor phase (top), costant fluid--rich phase (middle), Dirichlet boundary conditions fixing the two phases at the ends of the sample (bottom).}
\label{fig1}
\end{figure}
 
\begin{figure}
\begin{picture}(200,400)(15,0)
\put(0,300)
{
\resizebox{6.5cm}{!}{\rotatebox{0}{\includegraphics{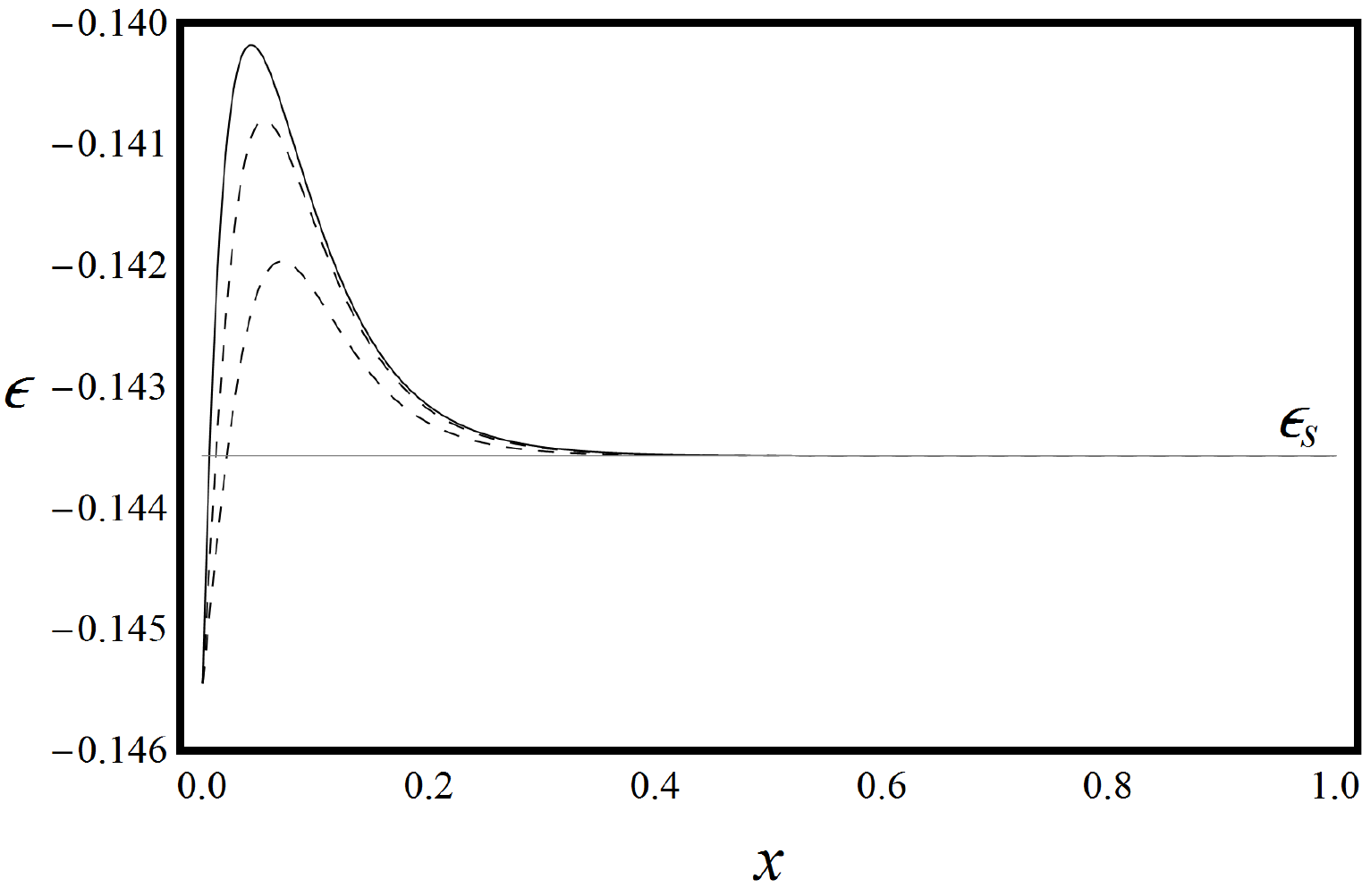}}}
}
\put(215,300)
{
\resizebox{6.5cm}{!}{\rotatebox{0}{\includegraphics{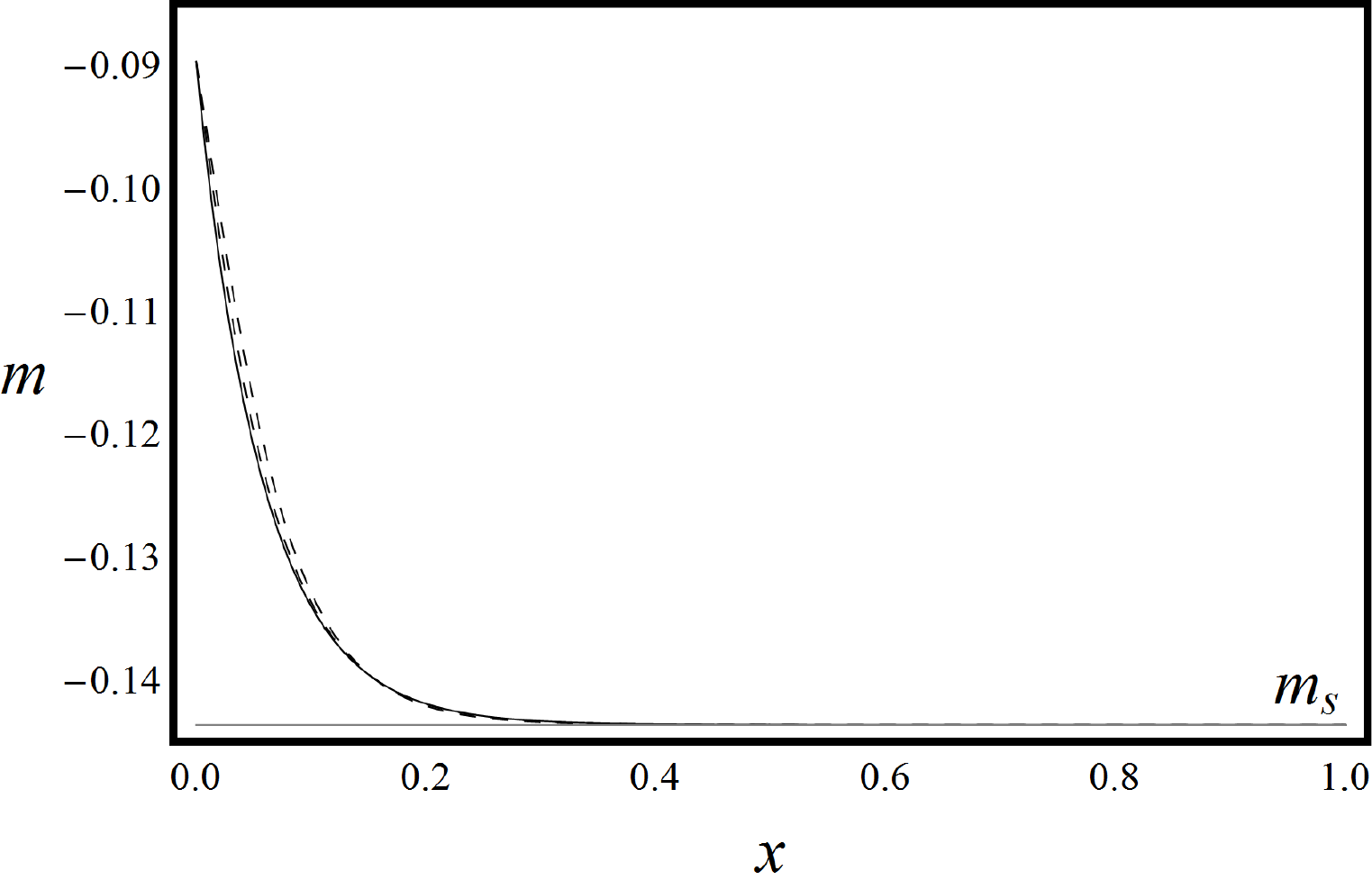}}}
}
\put(0,150)
{
\resizebox{6.5cm}{!}{\rotatebox{0}{\includegraphics{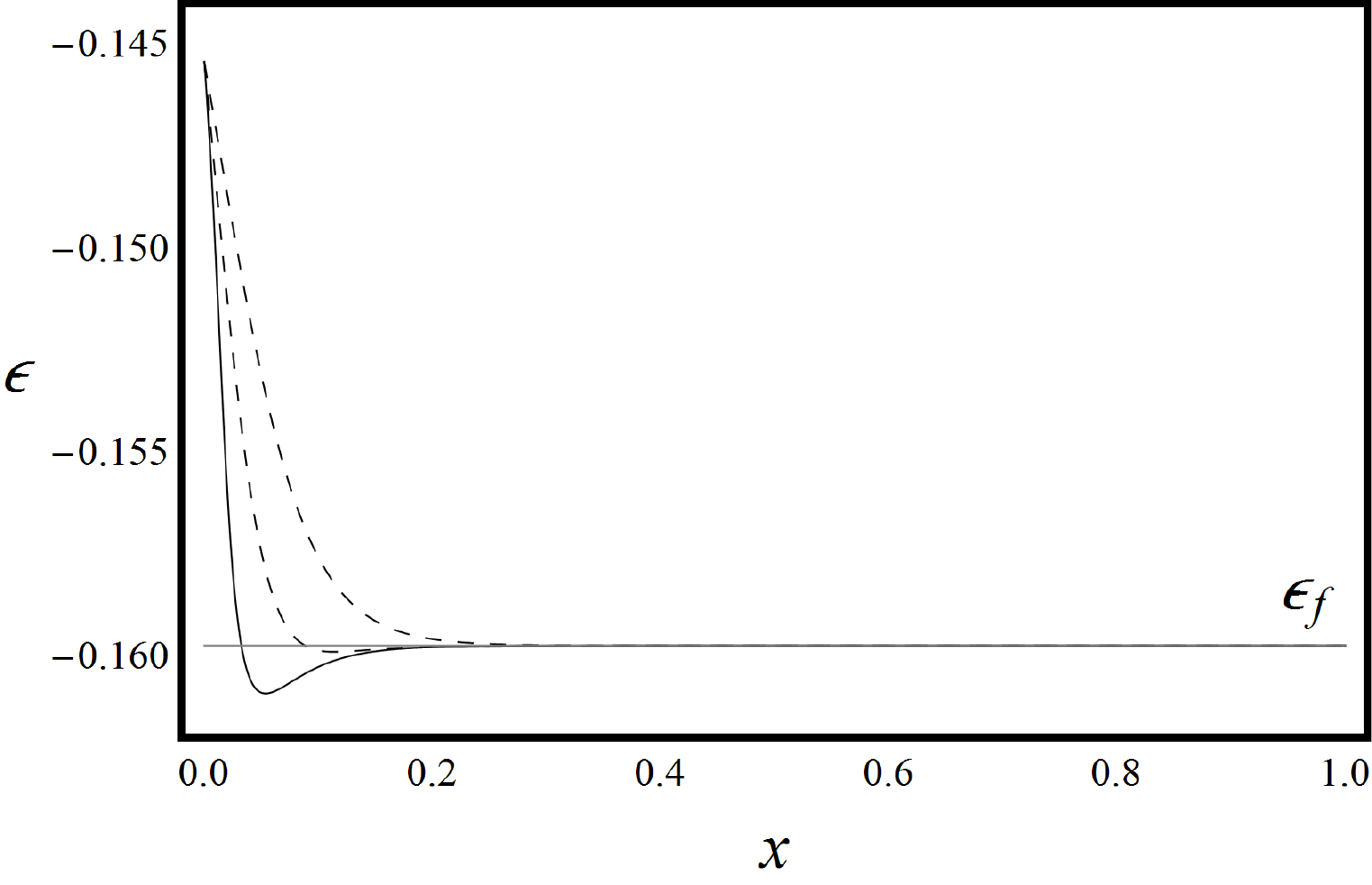}}}
}
\put(215,150)
{
\resizebox{6.5cm}{!}{\rotatebox{0}{\includegraphics{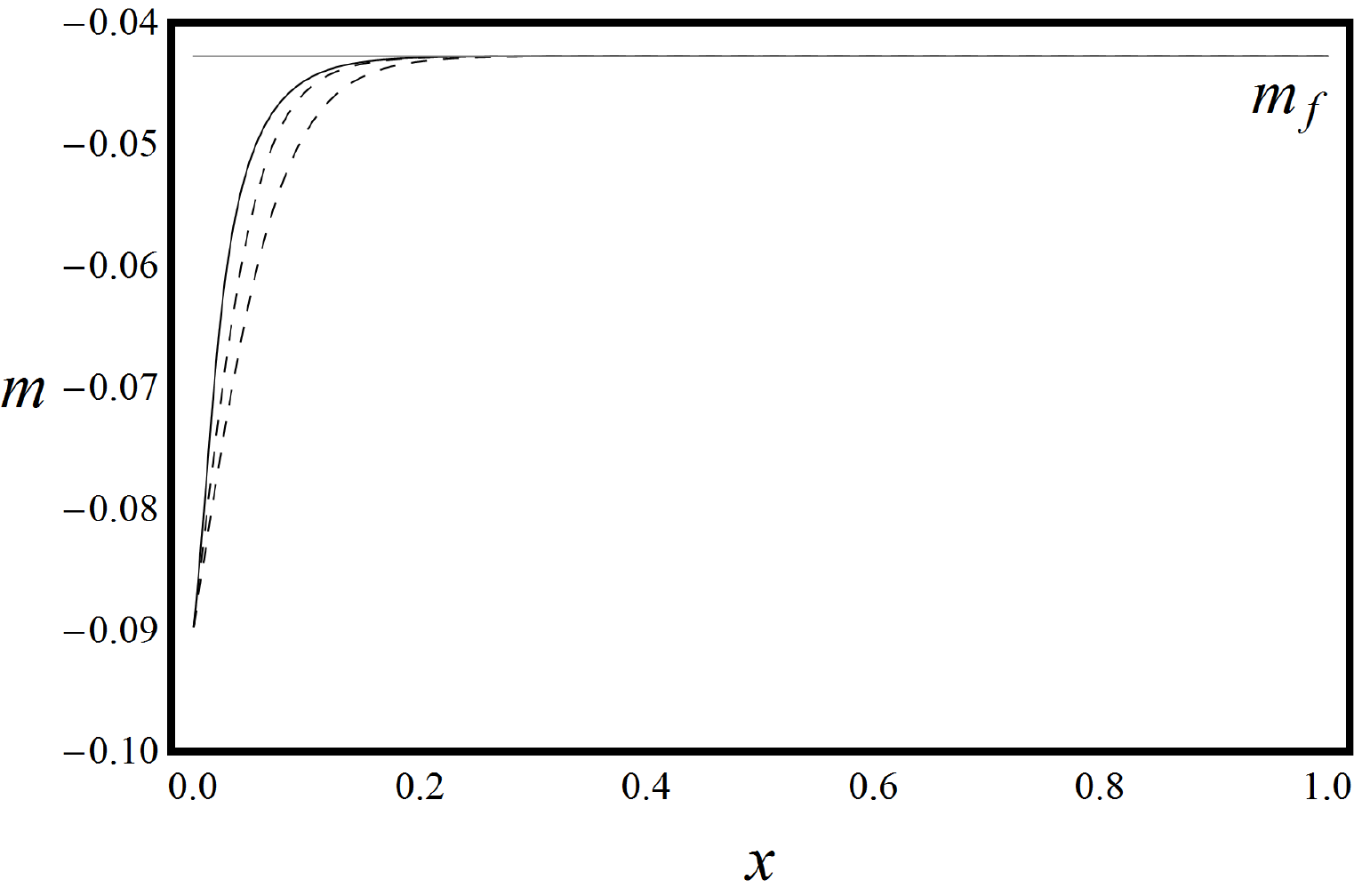}}}
}
\put(0,0)
{
\resizebox{6.5cm}{!}{\rotatebox{0}{\includegraphics{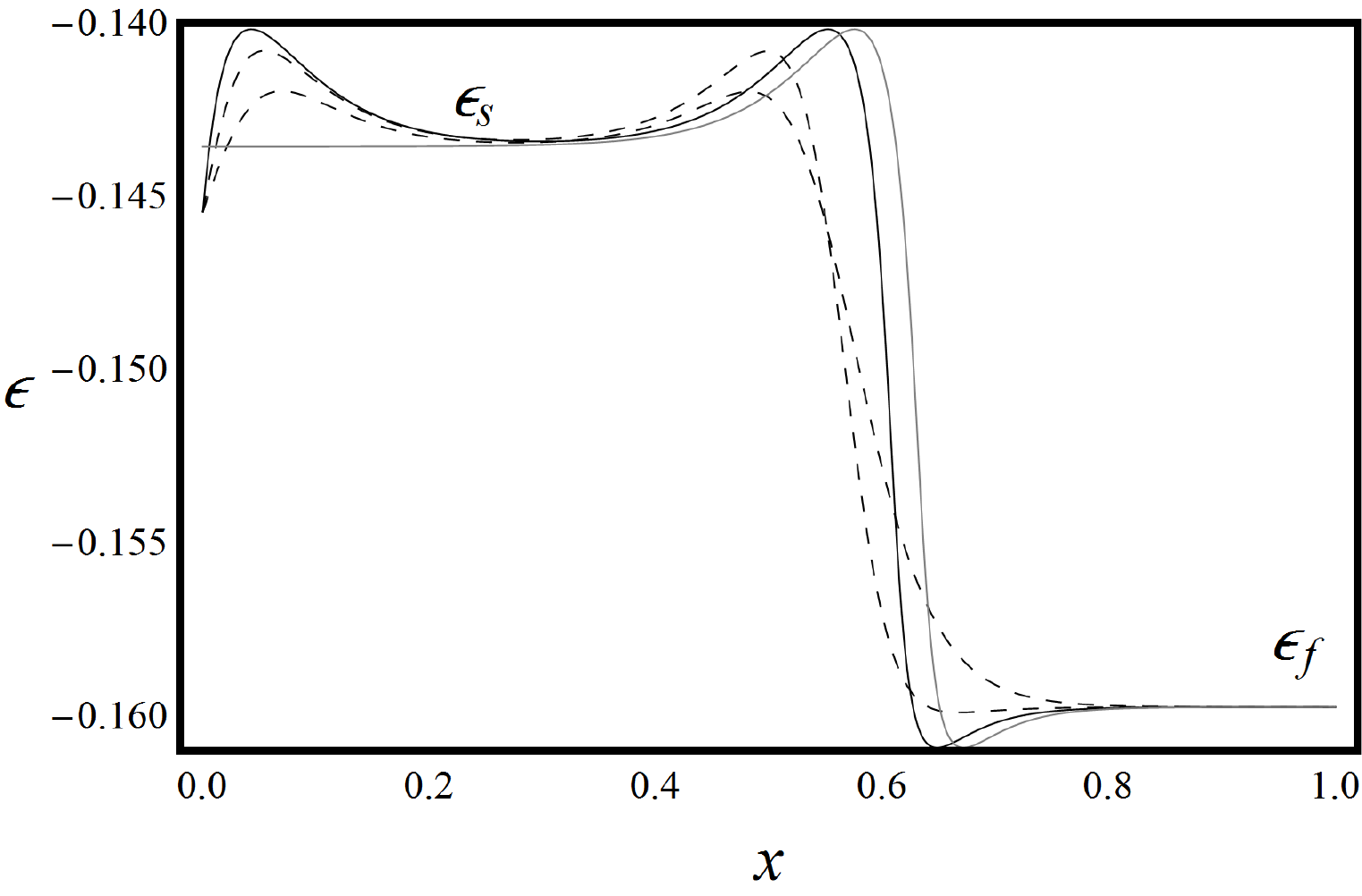}}}
}
\put(215,0)
{
\resizebox{6.5cm}{!}{\rotatebox{0}{\includegraphics{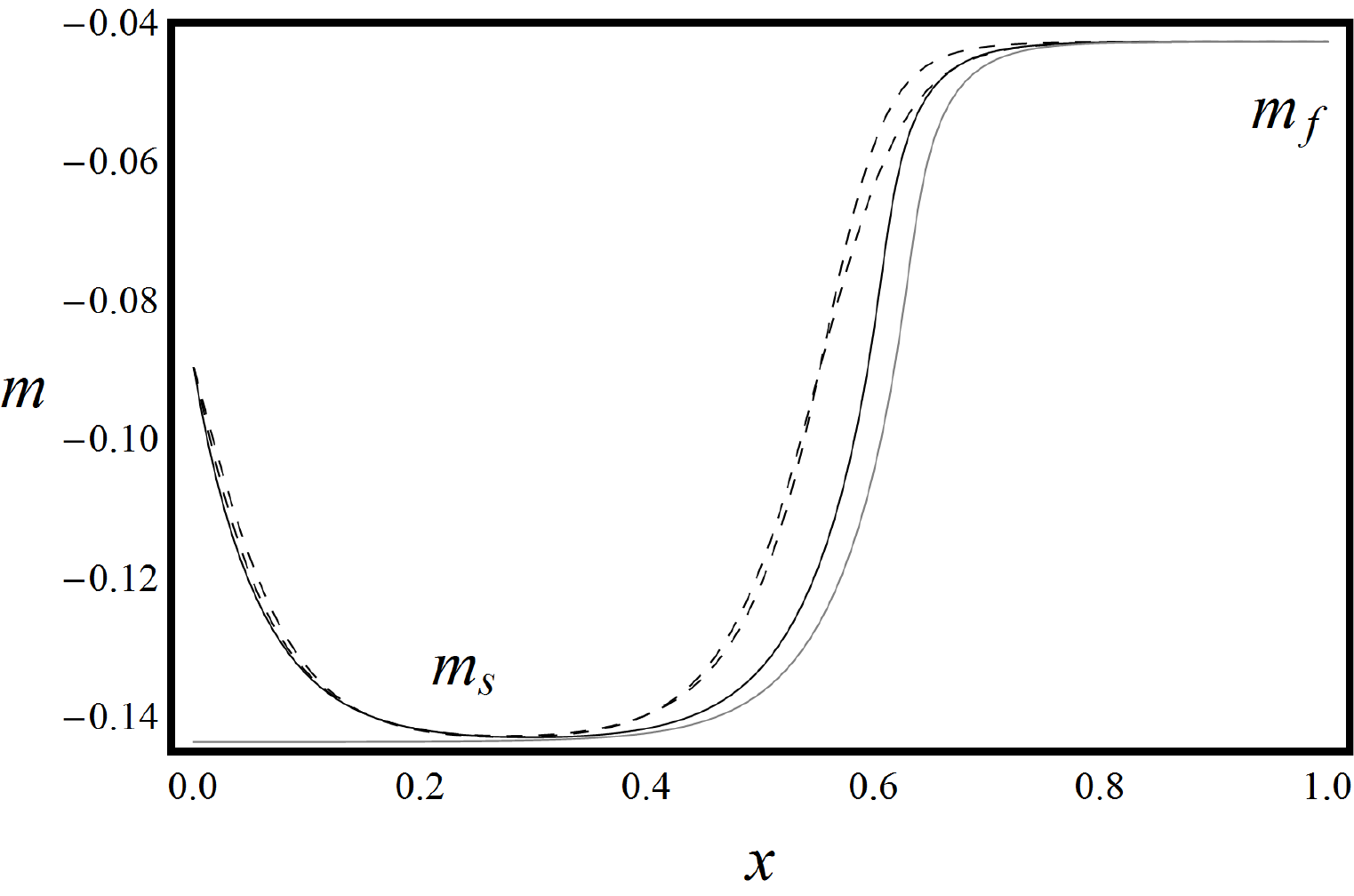}}}
}
\end{picture}
\vskip 1. cm
\centering
\caption{Solutions $\varepsilon(x)$ (left) and $m(x)$ (right) of the stationary problem \eqref{problemastazdipartenzaconepsilonpunto}--\eqref{8problemastazdipartenzaconepsilonpunto} with the boundary conditions $\varepsilon(0)=\bar{\bar{\varepsilon}}=-0.1454,\, m(0)=\bar{\bar{m}}=-0.0897,\,\partial_x \varepsilon(x=1)=0,\,and\,\partial_x m(x=1)=0$ on the finite interval $[0,1]$, for $p=0.24,\,a=0.5,\,b=1,\alpha=100,k_1=10^{-3}=10^{-3},\,k_2=10^{-3}$ (solid line), and $k_1=10^{-3}=10^{-3},\,k_2=0.2\times 10^{-3},\,0.8\times 10^{-3}$ (dotted lines), starting by the following initial guesses (gray lines): costant fluid--poor phase (top), costant fluid--rich phase (middle), Dirichlet boundary conditions fixing the two phases at the ends of the sample (bottom).}
\label{fig2}
\end{figure}

In figures \ref{fig1}--\ref{fig2}, the black solid lines correspond to the case $k_1=k_2=k_3=10^{-3}$, while the dashed lines correspond to the case $k_1=k_3=10^{-3}$ and $k_2=0.2\times 10^{-3},0.8\times 10^{-3}$. We recall that only in the second case the uniqueness of the solution to the time--dependent problem $(P)$ is ensured (see Theorem \ref{uniqueness}). We comment the physical features of the solution reffering to the fluid density profile $m$ (bottom of the figures). The charachteristics of the strain profile can be discussed accordingly.
\\
The three different graphs for $\varepsilon$ and $m$ in figure \ref{fig1} correspond to three different stationary solutions to problem \eqref{problemastazdipartenzaconepsilonpunto}--\eqref{8problemastazdipartenzaconepsilonpunto}, for the same Dirichlet boundary value $\varepsilon_D=\bar{\varepsilon},\,m_D=\bar{m}$ in $l_1=0$, where $\bar{\varepsilon},\bar{m}$ are values close to the fluid--poor phase $(\varepsilon_s,m_s)$ but slightly larger. 
In the top row the system is almost completely in the fluid--poor--phase, indeed in the interval $[0,0.2]$ the profile quickly decay from $\bar{m}$ to $m_s$ and then it stays constant. Note that no interface between the fluid--poor and the fluid--rich phase. In the central row, after a quick transition from $\bar{m}$ to $m_f$ the system constantly stays in the fluid--rich phase. Even in this case no interface is seen. Finally in the bottom row the stationary profile is an interface between the two phases $m_s$ and $m_f$. Indeed, the profile started at $\bar{m}$ first drops to $m_s$ and at a certain point quickly increases up to $m_f$.
\\
The three different graphs for $\varepsilon$ and $m$ in figure \ref{fig2} correspond to three different stationary solutions to problem \eqref{problemastazdipartenzaconepsilonpunto}--\eqref{8problemastazdipartenzaconepsilonpunto}, for the same Dirichlet boundary value $\varepsilon_D=\bar{\bar{\varepsilon}},\,m_D=\bar{\bar{m}}$ in $l_1=0$, where $\bar{\bar{\varepsilon}},\bar{\bar{m}}$ is the saddle point of the energy function $\Phi$. The essential features of the profiles are similar to those of figure \ref{fig1} but the choice of the value $\bar{\bar{m}}$ gives rise to small differences in the shape of the solutions.

The solutions of the stationary problem \eqref{problemastazdipartenzaconepsilonpunto}--\eqref{8problemastazdipartenzaconepsilonpunto} are obtained numerically via the finite difference method powered with the Newton--Raphson algorithm. The use of different initial guess in the Newton--Raphson algorithm has allowed us to find numerically the different stationary solutions. In particular, for both figures \ref{fig1}--\ref{fig2} in the top row we used as initial guess a costant function equal to the fluid--poor--phase, while in the central row it has been used a costant function equal to the fluid--rich--phase. Finally, in the bottom row we used as intial guess the solution to the same stationary problem with Dirichlet boundary conditions fixing the two phases at the ends of the sample, namely, $(\varepsilon(0),m(0))=(\varepsilon_s,m_s)$ and $(\varepsilon(1),m(1))=(\varepsilon_f,m_f)$ (see \cite{CIS2010}).
\\
\indent We now describe the adopted finite difference substitution rules. Let $n$ be a positive integer number and let $\sigma=1/n$ be the space increment. We subdivide the space interval $[0,1]$ into $n$ small intervals of lenght $\sigma$. Given a field $h(x)$, for any $i\in\{1,\dots,n-1\}$, we set
\begin{displaymath}
h'(i\sigma)\approx \frac{1}{2\sigma}[h((i+1)\sigma)
                    -h((i-1)\sigma)]
\end{displaymath}
For the second space derivative we set 
\begin{displaymath}
h''(i\sigma,)
\approx
\frac{1}{\sigma^2}[h((i+1)\sigma)-2h(i\sigma)
                    +h((i-1)\sigma)]
\end{displaymath}
for $i\in\{1,\dots,n-1\}$.

\section{Conclusions}
\label{s:conc}
We have studied the existence and the uniqueness of weak solutions  to 
the problem $(P)$ introduced in Section~\ref{Strong}. 
We have stressed that the mathematical interest 
of this problem lies on the coupled cross-diffusion-like structure of the transport fluxes. 
The problem has, also, a remarkable physical application 
in the framework of the Porous Media theory (see the discussion 
Section~\ref{s:cons}).

It is worth noting that 
our mathematical approach is restricted to the one--space dimension case (as far as we are 
concerned with the passage to the limit $\delta\to 0$) and cannot be 
extended for higher space dimensions in a natural way. To make progress in this direction 
we hope to be able to employ the hidden variational structure of the 
problem \cite[Section~2.4]{CIS2013}, see also \eqref{vardiss}.  
Regardless the choice of space dimension, we find mathematically interesting the study 
of 
the $t\to\infty$ asymptotics in the case when multiple steady states are 
expected. Similar considerations can be made in the Cahn--Hilliard setup. 

\vskip 0.5 cm
\par\noindent
\textbf{Acknowledgments}
\par\noindent
The authors thank G. Sciarra (Rome) for fruitful discussions on the consolidation problem. AM acknowledges useful discussions with T. Aiki (Tokyo) on related matters.
 


\end{document}